\documentclass[a4paper]{amsart}
\usepackage{amssymb, amsmath, amsthm}
\usepackage{stmaryrd, esint}
\usepackage{graphicx}
\usepackage{wrapfig}
\usepackage[pagewise]{lineno}
\DeclareGraphicsExtensions{.png}

\newtheorem{theorem}{Theorem}[section]
\newtheorem{lemma}[theorem]{Lemma}

\newtheorem{corollary}[theorem]{Corollary}

\theoremstyle{definition}
\newtheorem{definition}{Definition}[section]
\newtheorem{example}[definition]{Example}

\theoremstyle{remark}
\newtheorem{remark}[definition]{Remark}

\numberwithin{equation}{section}

\newcommand{\abs}[1]{\lvert#1\rvert}
\newcommand{\Abs}[1]{\left\lvert#1\right\rvert}

\newcommand{\tr}[1]{\bigr\vert_{\scriptscriptstyle{#1}}}
\newcommand{\A}{\mathcal{A}}

\newcommand{\G}{\mathcal{G}}
\newcommand{\h}{\mathcal{H}}
\newcommand{\I}{\mathcal{I}}

\newcommand{\R}{\mathbb{R}}

\newcommand{\C}{\mathbb{C}}

\newcommand{\N}{\mathbb{N}}
\newcommand{\Z}{\mathbb{Z}}

\DeclareMathOperator{\reg}{reg}
\DeclareMathOperator{\sing}{sing}

\DeclareMathOperator{\dist}{dist}
\DeclareMathOperator{\diam}{diam}
\title[Examples of holomorphic functions]{Examples of holomorphic functions vanishing to infinite order at the boundary}

\author[J.Hirsch]{Jonas Hirsch}

\begin{document}

\begin{abstract}
We present examples of holomorphic functions that vanish to infinite order at points at the boundary of their domain of definition. They give rise to examples of Dirichlet minimizing $Q$-valued functions indicating that "higher"-regularity boundary results are difficult. Furthermore we discuss some implication to branching and vanishing phenomena in the context of minimal surfaces, $Q$-valued functions and unique continuation.
\end{abstract}
\maketitle
\section*{Introduction} % (fold)
\label{sec_E:introduction}
In general branching phenomena are of interest in geometric measure theory, geometry and are strongly related to vanishing phenomena in the context of PDE's. There is a vast literature on branching results in the interior and one has plenty of unique continuation results for PDE's in the interior of their domains of definition. \\
One of the most fundamental examples is the following: two holomorphic functions that agree to infinite order at one point in the interior have to be identical.\\
Almgren's frequency function provides  a quite robust tool in order to capture branching (e.g. \cite{Almgren}) and unique continuation properties, (e.g. \cite{GaLi2}, \cite{GaLi}). Its monotone behaviour is crucial.

In this paper we construct examples which show that these interior results does not carry over to the boundary case:\\
There are holomorphic functions that vanish to infinite order at boundary points but which are not identically zero. Furthermore Almgren's frequency functions fails to be monotone at these points.

A further property of Almgren's frequency functions is that it enables a stratification procedure (compare the "dimension reduction" argument of Federer \cite{Federer}). As an outcome one was able to bound the Hausdorff dimension of the singular set of solutions to various PDE's. For instance \cite{Almgren}, \cite[section 3.4 - 3.6]{Lellis}, \cite{SimonWick}, \cite{KrummelWick} consider geometric problems, \cite{CheegerNaberValtorta}, \cite{NaVa1} analyse the singular set of elliptic PDEs.\\
Our examples suggest that such an upper bound fails to hold up to the boundary. \\

For instance there is a $Q$-valued function ($Q\ge 2$) on the half plane $\R^2_+$ that is Dirichlet minimizing with respect to compact deformations. It has a "smooth" trace on $\partial \R^2$. Nevertheless the Hausdorff dimension of the closure of the singular set is $1$. In contrast the singular set of an arbitrary $Q$-valued Dirichlet minimizer consists of isolated points in any proper subset of its domain of definition, compare \cite[Theorem 0.12]{Lellis}. More details and the precise statement can be found in section \ref{sub:Q-valued functions}.\\

The main ingredient of the presented examples to branching phenomena of minimal surfaces, $Q-$valued Dirichlet minimizers and unique continuations results are examples of holomorphic functions on the half plane $\C_+ = \{ z \in \C \colon \Re(z)>0 \}$ that admit $C^\infty$-extension to $\overline{\C_+}$with the following additional properties:\\

\begin{lemma}\label{lem_E:1.1}
	Let $0<s\le 1$ be given. There exist 
	\begin{itemize}
		\item[(i)] a nowhere dense compact Cantor type  subset $E_s\subset [0,1]$ with $\h^s(E_s)=1$ if $0<s<1$ and $dim_{\h}(E_1)=1$;
		\item[(ii)] holomorphic functions $F(z)$, $G(z)$ on $\C_+$ with the property that $f(z)=e^{-F(z)}$, $g(z)= G(z)e^{-F(z)}$ admit $C^\infty$-extensions to $\overline{\C_+}$. Moreover, $f,g$ vanish to infinite order at any $z \in -iE_s$ and for every $z \in -iE_s$ there is a sequence $z_k \in \C_+$ with $z_k \to z$ and $g(z_k)=0$ for all $k$. 
	\end{itemize}
\end{lemma}
These functions are constructed similar to the Weierstrass' function, an example of a continuous but nowhere differentiable function. Instead of an infinite series we use infinite products of the following holomorphic \emph{building blocks}:
\begin{align}\label{eq:0.101}
	a(z)&=e^{-z^{-\alpha}} \text{ for } 0<\alpha <1\\\nonumber
	b(z)&=\cos(\ln(z)) e^{-z^{-\alpha}} \text{ for } 0<\alpha<1.
\end{align}\\

The paper is organized as follows:\\
In section \ref{sec:construction_of_the_sets_e_s_} we present a Cantor type set $E_s$ that will be the "singular"/"vanishing" set of the holomorphic functions.\\
In section \ref{sec:construction_of_the_holomorphic_functions} the holomorphic functions are constructed and their claimed properties are proven.\\
Finally in section \ref{sec:applications} we present applications of these functions to branching results of minimal surfaces, $Q$-valued functions and unique continuation type results.
% section introduction (end)
\section*{Acknowledgements}
My most sincere thanks go to Camillo De Lellis and Emanuele Spadaro for their insights and stimulating discussions. Their knowledge and expertise on a vast number of topics was invaluable. Moreover I want to thank Neshan Wickramasekera for the initial idea. Tobias Lamm for helpful and encouraging feedback. Finally, I thank the many other people who I have failed to mention in this abbreviated list and who helped me along the way.

%\tableofcontents

\section{Construction and properties of the set $E_s$} % (fold)
%\label{sec:construction_and_properties_of_the_set_e_s_}
\label{sec:construction_of_the_sets_e_s_}
The construction of $E_s$ is a classical Cantor type construction. Nonetheless for the sake of completeness and to fix certain parameters we present the construction in detail. We follow closely an approach of Falconer in \cite[Theorem 1.15]{Falconer}. We will use the following the definition of the ("unnormalised") Hausdorff measure, \cite[Section 1.2]{Falconer}:
Given a subset $E \subset \R^n$, a (countable) family $\mathcal{U}=\{U_j\}_{j \in \mathcal{J}}$ is called $\delta$-cover of $E$, if $E\subset \bigcup_{U \in \mathcal{U}} U$ and $\diam(U)<\delta$. We define
\[ \mathcal{H}^s_\delta(E):=\left\{ \sum_{U \in \mathcal{U}} \diam(U)^s \colon \mathcal{U} \text{ $\delta$-cover of $E$}\right\}.\]
$\mathcal{H}^s_\delta$ defines an outer measure on $\R^n$ and the $s$-dimensional Hausdorff measure is defined to be
\[ \mathcal{H}^s(E):= \sup_{\delta>0} \mathcal{H}_\delta^s(E). \]
A subset $E$ is said to have Hausdorff dimension $s$ if $s=\sup\{ \sigma>0 \colon  \mathcal{H}^{\sigma}(E)=+\infty \}$.

\begin{lemma}\label{lem_E:2.1}
Let $0<s\le 1$ be given. Then there is a nowhere dense compact subset $E_s \subset [0,1]$ s.t. $\mathcal{H}^s(E_s)=1$ if $0<s<1$ and $dim(E_1)=1$. 
\end{lemma}
\begin{proof}
The set $E_s$ is obtained classically as the intersection of a decreasing family of compact sets 
\[
	E_s = \bigcap_{k=1}^\infty \bigcup_{l=1}^{2^k} E_{k,l}.
\]
The compact subintervals $E_{k,l}$ are defined inductively.

We fix a sequence of parameters by 
\[
	\frac{1}{\sigma_k}=\begin{cases}
		\frac{1}{s} , &\text{ if } 0< s< 1\\
		1+k^{\frac{2}{3}}- (k-1)^{\frac{2}{3}} , &\text{ if } s=1
	\end{cases}
\]
In both cases we have $\sigma_k\le \sigma_{k+1}$ and $\sigma_k<1$ all all $k$ . If $s=1$ we have $\frac{1}{\sigma_{k+1}}-\frac{1}{\sigma_k} = (k+1)^\frac{2}{3}+(k-1)^\frac{2}{3} - 2 k^\frac{2}{3} <0$ due to concavity of $t \mapsto t^{\frac{2}{3}}$ and so $\sigma_k \nearrow 1$ as $k \to \infty$.

We choose $E_{0,1}=[0,1]$ and proceed inductively. Suppose $E_{k-1, l}, l=1, \dotsc, 2^{k-1}$ defined, then $E_{k, 2l-1}, E_{k, 2l}$ are the closed disjoint subintervals obtained by removing an open interval in the middle of $E_{k-1,l}$ with
\begin{equation}\label{eq_E:2.101}
	\abs{E_{k,2l-1}}^{\sigma_{k}}=\abs{E_{k,2l}}^{\sigma_{k}}=\frac{1}{2} \abs{E_{k-1,l}}^{\sigma_{k}}.
\end{equation}
(Which is possible since $\sigma_k<1$.)
We obtained $2^k$ closed intervals $E_{k,l}$ of equal length
\begin{equation}\label{eq_E:2.102}
	\abs{E_{k,l}}= 2^{-\frac{1}{\sigma_k}} \abs{E_{k-1, l'}} = \begin{cases}
		2^{-\frac{k}{s}}, &\text{ if } 0< s < 1\\
		2^{-k-k^{\frac{2}{3}}}, &\text{ if } s=1
	\end{cases}
\end{equation}
where we used that $\sum_{l=1}^k \sigma_k^{-1}=\frac{k}{s}$ if $0<s<1$ and $\sum_{l=1}^k \sigma_k^{-1}= k + k^{\frac{2}{3}}$ if $s=1$.\\

In a first step we will check that $\h^s(E_s)\le 1$ $(\h^1(E_1)=0)$. To do so, let $\delta>0$ be given. Due to \eqref{eq_E:2.102} there is $k_0>0$ with $\abs{E_{k_0,l}}< \delta$. Hence $\{E_{k,l}\}_{l=1}^{2^k}$ is an admissible $\delta$-cover for $E_s$ for any $k\ge k_0$. With \eqref{eq_E:2.102} in mind we have
\begin{equation}\label{eq_E:2.103}
	\h^s_\delta(E_s) \le \sum_{l=1}^{2^k} \abs{E_{k,l}}^s = \begin{cases}
		2^k \left( 2^{-\frac{k}{s}} \right)^s = 1,  &\text{ if } 0<s<1\\
		2^k 2^{-k - k^{\frac{2}{3}}} \to 0, &\text{ if } s=1, k \to \infty.
	\end{cases}
\end{equation}\\

Now in the second step we check that $\h^s(E_s) \ge 1$ if $s<1$ and $\h^\sigma(E_1)=+\infty$ for all $\sigma<1$ if $s=1$. Equivalently we have to show that for any given $\epsilon>0$, $\sigma<1$ there is a $\delta>0$ with the property that for any $\delta-$cover $\mathcal{U}$ of $E_s$ 
\begin{align}\label{eq_E:2.104}
	\sum_{C \in \mathcal{U}} \diam(C)^{s}&\ge \h_\delta^s(E_s) > 1-\epsilon, &&\text{ if } 0<s<1\\ \nonumber
	\sum_{C \in \mathcal{U}} \diam(C)^{\sigma}&\ge \h_\delta^\sigma(E_1) > \frac{1}{\epsilon}. &&\text{ if } s=1 \text{ i.e. } \sigma<1
\end{align}
Let $\epsilon >0$, $\sigma<1$ be given. We fix $k_0>0$ large, determined later s.t. at least $\sigma_{k_0}>\sigma$ and $0< \delta< \abs{E_{k_0,l}}$.\\

Fix an admissible $\delta-$cover $\mathcal{U}$ by intervals $E_{k,l}$. Hence $k>k_0$ for any of these intervals. The compact intervals $E_{k,l}$ are relative open to the compact set $E_s$, so that the cover can assumed to be finite. Removing all intervals that are contained in some other of the collection we can even assume that they are mutually disjoint. 
Let $E_{k,2l-1}$ (or $E_{k,2l}$) be one of the shortest intervals in $\mathcal{U}$. Its companion $E_{k,2l}$ (respectively $E_{k,2l-1}$) has to be in $\mathcal{U}$ as well because all intervals are disjoined and they are one of shortest. The sums in \eqref{eq_E:2.104} do not increase if we replace these two intervals by its predecessor $E_{k-1,l}\supset E_{k,2l-1} \cup E_{k,2l}$ because
\begin{align*}
	\abs{E_{k,2l-1}}^s + \abs{E_{k,2l}}^s &= \abs{E_{k-1,l}}^s, &&\text{ if } 0<s<1\\
	\abs{E_{k,2l-1}}^\sigma + \abs{E_{k,2l}}^\sigma &= 2^{1-\frac{\sigma}{\sigma_k}} \abs{E_{k-1,l}}^\sigma \ge \abs{E_{k-1,l}}^\sigma, &&\text{ if } s=1 \text{ i.e. } \sigma<1
\end{align*}
where we used \eqref{eq_E:2.101} and $\sigma_k \ge \sigma_{k_0}> \sigma$.
We may proceed in this way, replacing the shortest intervals by larger ones without increasing the value of the sums, until we reach that all intervals are of same size i.e. $\mathcal{U} \to \{ E_{k_1, l}\}_{l=1}^{2^{k_1}}$ for some $k_1 > k_0$. We conclude
\begin{align*}
	\sum_{C \in \mathcal{U}} \diam(C)^{s}&\ge \sum_{l=1}^{2^{k_1}} \abs{E_{k_1,l}}^s = 1, &&\text{ if } 0<s<1\\ \nonumber
	\sum_{C \in \mathcal{U}} \diam(C)^{\sigma}&\ge \sum_{l=1}^{2^{k_1}} \abs{E_{k_1,l}}^\sigma = 2^{(1-\sigma)k_1 - \sigma k_1^{\frac{2}{3}}} > \frac{1}{\epsilon}. &&\text{ if } s=1 \text{ i.e. } \sigma<1
\end{align*}
where we used \eqref{eq_E:2.103} and $2^{(1-\sigma)k_1 - \sigma k_1^{\frac23}}\to \infty$ as $k_1 \to \infty$.\\
%$k_1>k_0$ with $k_0>0$ sufficient large s.t. $2^{(1-\sigma)k_0 - \sigma k_0^{\frac{2}{3}}} > \frac{1}{\epsilon}$.\\

It remains to argue that the assumption that the $\delta-$cover is made out of intervals $E_{k,l}$ is no restriction. Fix  any $\delta$-cover $\mathcal{V}$. We can assume that it consists of open intervals without changing the value in \eqref{eq_E:2.104} significantly. Since $E_s$ is compact the cover can assumed to be finite.\\

Firstly let us argue for $E_1$. Any interval $I \in \mathcal{V}$ intersects at most three intervals $E_{k_I,l}$ with $\abs{E_{k_I,l}} \le \abs{I} < \abs{E_{k_I-1,l}}$. Otherwise $I$ would need to contain an interval of length at least $\abs{E_{k_I-1,l}}$ due to the Cantor type construction. This is impossible by the choice of $k_I$. Replacing $I$ by these at most three intervals $E_{k_I, \cdot}$ and the same for any other interval in $\mathcal{V}$ we obtain an open cover $\mathcal{U}$ by intervals $E_{k,l}$. Furthermore
\[
	\sum_{E_{k,l} \in \mathcal{U}} \abs{E_{k,l}}^\sigma \le 3 \sum_{I \in \mathcal{V}} \abs{I}^\sigma.
\]
We had just shown that the left hand side is larger then $\frac{1}{\epsilon}$, so \eqref{eq_E:2.104} holds for $s=1$.

Secondly, we argue for $E_s$, $0<s<1$ as follows: Recursively we will change $\mathcal{V}$ into a $\delta$-cover $\mathcal{U}$ by sets $E_{k,l}$ without increasing the sum in \eqref{eq_E:2.104}.  Since $\R \setminus E_s$ is open dense, we may assume that $\partial I \cap E_s = \emptyset$ for all $I \in \mathcal{V}$ without changing the sum in \eqref{eq_E:2.104} significantly. Replacing each $I$ by $J=\overline{I}\cap[0,1]$ we obtain a finite cover by closed sets, with the additional properties:  for each $I$ we have $\partial J \cap \left(E_s \setminus \{0,1\}\right) = \emptyset$, $J\subset [0,1]=E_{0,1}$ and for all $k$ sufficient large (depending on I) we have
\begin{equation}\label{eq.terminate} E_{k,l} \cap J = E_{k,l} \text{ or } E_{k,l} \cap J =\emptyset.	
\end{equation}
Given one of these intervals $J$ and let $J \subset E_{k-1,l}$ for some $k,l$. Then 
\begin{equation}\label{eq_E:2.105}
	\abs{J\cap E_{k,2l-1}}^s + \abs{J\cap E_{k,2l}}^s \le \abs{J \cap E_{k-1,l}}^s
\end{equation}
because $\abs{ E_{k,2l-1}}^s + \abs{ E_{k,2l}}^s = \abs{E_{k-1,l}}^s$ and the left-hand side of \eqref{eq_E:2.105} increases faster then the right-hand side. If either $J\cap E_{k,2l-1}\neq E_{k,2l-1}$ or $J\cap E_{k,2l}\neq E_{k,2l}$, we repeat the process, replacing $J\cap E_{k,2l-1}$ and $J\cap E_{k,2l}$ by smaller intervals. This process terminates after finitely many steps due to \eqref{eq.terminate}. Finally we obtained the desired cover $\mathcal{U}$. By construction we ensured $\sum_{J \in \mathcal{ V}_1} \abs{J}^s \ge \sum_{E_{k,l} \in \mathcal{U}} \abs{E_{k,l}}^s=1$. This proves \eqref{eq_E:2.104} if $0<s<1$.\\
% This observation proves with the analysis above $\mathcal{H}^s(E_s)=1$.
\end{proof}

% section construction_and_properties_of_the_set_e_s_ (end)

\section{construction of the holomorphic functions} % (fold)
\label{sec:construction_of_the_holomorphic_functions}
The Cantor set $E_s$ was obtained as 
\[
	E_s = \bigcap_{k=1}^\infty \bigcup_{l=1}^{2^k} E_{k,l}.
\]
Based on this construction, we define the index set:
\[
	\mathcal{I}=\left\{ (k,l)\colon k=1, \dotsc, \infty, l=1, \dotsc, 2^k \right\} \text{ with } \tau=(k,l) \in \mathcal{I}. 
\]
Recall that the enumeration had been chosen s.t. $E_{k,2l-1} \cup E_{k,2l} \subset E_{k-1,l} \forall (k,l)$.
The Cantor set $E_s$ constructed in lemma \ref{lem_E:2.1} has the property \eqref{eq_E:2.102} 
\begin{equation*}
	\abs{E_\tau}=\abs{E_{k,l}} = \begin{cases}
		2^{-\frac{k}{s}}, &\text{ if } 0< s < 1\\
		2^{-k-k^{\frac{2}{3}}}, &\text{ if } s=1
	\end{cases} \quad \forall \tau \in \I.
\end{equation*}
We denote with $y_\tau$ the left boundary point of the compact interval $E_\tau$. Note that the construction of $E_s$ ensures that $\forall z \in E_s$ there exists a sequence $\{y_{\tau(k)}\}_{k \in \N}$ such that $y_{\tau(k)} \to z$, i.e. $\{y_\tau\}_{\tau \in \I}$ is dense in $E_s$. %TODO sequence

Furthermore it is useful to fix some terminology. $\R_-=\{ z=x+i0 \colon x < 0 \}$ denotes the negative real axis. We will use $z+iy_\tau= r_\tau e^{i\theta_\tau}$ for any $\tau \in \I$. And for any $y \in \R$ let $\R_--iy$ be the by $-iy$ translated negative real axis i.e. the set $\{ x-iy \colon  x < 0 \}$. And we will use \[\R_- - iE_s= \bigcup_{y \in E_s} (\R_- -i y) =\{ x-iy \colon x \in \R_-, y \in E_s \}.\]

The proof to lemma \ref{lem_E:1.1} is split into two parts. In the next paragraph we construct holomorphic functions $F,G$ based on the Cantor set $E_s$ and then in the subsequent paragraph the $C^\infty$ extension is proven. 

\subsection{Holomorphy} % (fold)
\label{sub:holomorphy}
On the slit plane $\C\setminus \R_-$ the principal value of the logarithmic function $\ln: \C\setminus \R_- \to \C \cap \{ -\pi < \Im(z)< \pi \} $ is single valued and holomorphic. So will be all roots for $\alpha \in \R$ defined as $z^{\alpha}=e^{\alpha \ln(z)}$.\\
As composition of holomorphic functions on $\C\setminus \R_-$ the building blocks, $a(z)=e^{-z^{-\alpha}}, b(z)=\cos(\ln(z))e^{-z^{-\alpha}}$ are clearly holomorphic on $\C\setminus\R_-$\\
$(z+iy_\tau)^{-\alpha}=r_{\tau}e^{-i\alpha \theta_\tau}$ is single valued and holomorphic on $\C \setminus (\R_- - i y_\tau) \subset \C\setminus(\R_- - i E_s )$ for every $\tau \in \I$, $\alpha_k \in \R$.\\

\begin{lemma}\label{lem_E:3.1}
	Given a sequence of complex numbers $a_k \in \C$ with $\sum_{k=0}^\infty 2^k \abs{a_k}< \infty$ and a sequence of real numbers $0<\alpha_k\le 1$ then
	\[
		F(z)= \sum_{\tau \in \I} a_k (z+ i y_\tau)^{-\alpha_k}
	\]
	is holomorphic on $\C\setminus \{\R_--iE_s\}$ and so is $e^{-F(z)}$.
\end{lemma}

\begin{proof}
	For a fixed $0<d<1$ we have for any $z \in \{ z\in C \colon \dist(z, -iE_s)> d\}$ satisfies $\abs{(z+iy_\tau)^{-\alpha_k}}= r_\tau^{-\alpha_k} \le d^{-1}$. So that the sum $\sum_{\tau\in \I} \abs{a_k(z+iy_\tau)^{-\alpha_k}} \le d^{-1} \sum_{k=1}^\infty 2^k \abs{a_k}<\infty$
	converges absolutely. $F$ is therefore the uniform limit of holomorphic functions on $\{ z\in C \colon \dist(z, -iE_s)> d\}$ and so itself holomorphic. $d$ has been arbitrary and therefore $F $ is holomorphic on $\C\setminus (\R_--iE_s)$. $e^{-F(z)}$ is the composition of two holomorphic functions and so itself holomorphic on the same set. 
\end{proof}

\begin{lemma}\label{lem_E:3.2}
	Given a sequence of non-negative real numbers $b_k \in \R_+$ that satisfies $\sum_{k=0}^\infty 2^k b_k <\infty$, then for any subset $\mathcal{J}\subset \I$ 
	\begin{equation}\label{eq:3.101}
		G_\mathcal{J}(z)= \prod_{\tau \in \mathcal{J}} \cos(b_k\ln(z+iy_\tau))
	\end{equation}
	is holomorphic on $\C\setminus (\R_- - i E_s )$ and uniformly bounded by
	\[
		\abs{G_\mathcal{J}(z)} \le e^{\sum_{\tau \in \mathcal{J}} b_k \abs{\theta_\tau}} \le e^{\pi \sum_{k =0}^\infty 2^k b_k}.
	\]
\end{lemma}
\begin{proof}
As a composition of holomorphic functions $\cos(b_k \ln(z+iy_\tau))$ is holomorphic on $\C \setminus (\R_--iE_s)$ for every $\tau \in \I$. Using the expansion
\begin{equation}\label{eq:cos_expansion}
	\cos(x+iy)=\cos(x)\cosh(y)-i\sin(x)\sinh(y)
\end{equation}
we have 
\begin{equation}\label{eq:cos_expansion_evalutated}
	\cos(b_k\ln(z+iy_\tau))=\cos(b_k \ln(r_\tau))\cosh(b_k\theta_\tau)+ i \sin(-b_k \ln(r_\tau))\sinh(b_k \theta_\tau)
\end{equation}
For every $\tau \in \I$ we have therefore
\begin{align}
	&\abs{\cos(b_k\ln(r_\tau))\cosh(b_k\theta_\tau)}\le \abs{\cos(b_k\ln(z+iy_\tau))} \le \cosh(b_k\theta_\tau)\label{eq:3.102}\\
	&\frac{\Im(\cos(b_k\ln(z+iy_\tau)))}{\Re(\cos(b_k\ln(z+iy_\tau)))}= \tan(-b_k \ln(r_\tau))\tanh(b_k\theta_\tau). \label{eq:3.103}
\end{align}
To show that \eqref{eq:3.101} is well defined and holomorphic, fix $0<d<\frac12$ and $k_0 \in \N$ sufficient large s.t. $0\le-2 b_k\ln(d)\le \frac{\pi}{4}$ for all $k \ge k_0$. This ensures that for any $z \in \{ d < \dist(z, -iE_s) < \frac{1}{d}\}$ and $\tau \in \I \cap \{k\ge k_0\}$ we have $d\le r_\tau \le \frac{1}{d} + \operatorname{diam}(E_s) \le \frac{1}{d^2}$. Hence $-\frac{\pi}{4}< b_k \ln(r_\tau)< \frac{\pi}{4}$ and so $\Re\left(\cos(b_k\ln(z+iy_\tau))\right) >0$. This implies that $\ln(\cos(b_k\ln(z+iy_\tau)))$ is a holomorphic function on $\{ d < \dist(z, -iE_s) < \frac{1}{d}\}$ if $\tau \in I \cap \{k\ge k_0\}$. Using \eqref{eq:3.102} we obtain 
\begin{equation}\label{eq:3.104}
	\ln(\cosh(b_k \theta_\tau))+ \ln(\cos(b_k \ln(r_\tau))) \le \ln(\abs{\cos(b_k\ln(z+iy_\tau))}) \le \ln(\cosh(b_k \theta_\tau)).
\end{equation}
This is the real part of $\ln(\cos(b_k\ln(z+iy_\tau)))$.
Its imaginary part, the argument of $ \cos(b_k\ln(z+iy_\tau))$ can be estimated by $\abs{b_k\ln(r_\tau)}$.  This follows from \eqref{eq:3.103} taking into account that $\abs{\tanh}<1$ and that $\tan(x)$ is convex on $[0, \frac{\pi}{4}]$ hence $\tan(s)\le s$ %$-\frac{\pi}{4}< b_k \ln(r_\tau)<\frac{\pi}{4}$.
Combining both we deduce
\[
	\abs{\ln(\abs{\cos(b_k\ln(z+iy_\tau))})} \le \ln(\cosh(b_k \theta_\tau)) - \ln(\cos(b_k \ln(r_\tau))) + \abs{b_k \ln(r_\tau)}.
\]
Furthermore we can use $\cosh(x)\le e^x$ and that $-\ln(\cos(x))\le C \abs{x}$ for some $C>0$ for $\abs{x}\le \frac{\pi}{4}$\footnote{This can be seen auf follows: $\psi(x)=\ln(\frac{1}{\cos(x)})$ as a composition of non-decreasing convex functions is convex on $]0, \frac{\pi}{2}[$. So $h(x)=\frac{\psi(x)}{x}$ is monoton increasing. Therefore we have $h(x) \le h(\frac{\pi}{4}) x$.} to estimate further: 
%One checks that $h(x)=\frac{-\ln(\cos(x))}{x}$ is monotone increasing on $]0,\frac{\pi}{2}[$, hence for $\abs{x}\le \frac{\pi}{4}$ we have $-\ln(\cos(x))\le C \abs{x}$ with $C=h(\frac{\pi}{4})$. Consequently we have
\[
	\abs{\ln(\abs{\cos(b_k\ln(z+iy_\tau))})} \le b_k \abs{\theta_\tau}+ (C+1) \abs{b_k \ln(r_\tau)} \le (\pi -2\ln(d) (C+1) ) b_k;
\]
$\sum_{\tau \in \I\cap \{k\ge k_0\}} \abs{\ln(\cos(b_k \ln(z+iy_\tau)))}< (\pi -2\ln(d) (C+1) ) \sum_{k=k_0}^\infty 2^k b_k$ converges uniformly on $\{ d < \dist(z, -iE_s) < \frac{1}{d}\}$ so that 
\[
	G_1(z)=e^{\sum_{\tau\in \mathcal{J},k \ge k_0} \ln(\cos(b_k \ln(z+iy_\tau)))}
\]
is holomorphic on $\{ d < \dist(z, -iE_s) < \frac{1}{d}\}$. In \eqref{eq:3.104} we observed that $\Re(\ln(\cos(b_k \ln(z+iy_\tau))))\le  \ln(\cosh(b_k \theta_\tau)) \le b_k \abs{\theta_\tau}$ and therefore
\[
	\abs{G_1(z)}=e^{\sum_{\tau\in \mathcal{J}, k \ge k_0} \Re(\ln(\cos(b_k \ln(z+iy_\tau))))} \le e^{\sum_{\tau\in \mathcal{J}, k \ge k_0} b_k \abs{\theta_\tau}}.
\]
\[
	G_2(z)=\prod_{\substack{\tau\in \mathcal{J}\\k<k_0}} \cos(b_k \ln(z+iy_\tau))
\]
is the product of finitely many holomorphic functions on $\C\setminus (\R_--iE_s)$ and so itself holomorphic with
\[
	\abs{G_2(z)} \le \prod_{\substack{\tau\in \mathcal{J}\\k<k_0}} \abs{\cos(b_k \ln(z+iy_\tau))} \le \prod_{\substack{\tau\in \mathcal{J}\\k<k_0}} \cosh(b_k \theta_\tau) \le e^{\sum_{\tau\in \mathcal{J},k<k_0} b_k \abs{\theta_\tau}}
\]
where we used \eqref{eq:3.102}. Multiplication of $G_1$ and $G_2$ closes the argument.
\end{proof}

We note that
$\cos(b_k \ln(z+iy_\tau))=0$ for $z= -iy_\tau + e^{-\frac{m\pi-\frac{\pi}{2}}{b_k}}$ for any $\tau=(k,l) \in \I$ and $m \in \N$, so that 
\begin{equation}\label{eq:zeros} G(z)=G_{\I}(z)=0 \text{ for all } z= -iy_\tau + e^{-\frac{m\pi-\frac{\pi}{2}}{b_k}}, \tau=(k,l) \in \I, m \in \N.\end{equation}
Consequently we got the following:

\begin{corollary}\label{cor_E:3.3}
	Let $\alpha_k,a_k,b_k$ be sequences of non-negative real numbers, that satisfies $0\le \alpha_k \le 1$ and $\sum_{k=1}^\infty 2^k a_k, \sum_{k=1}^\infty 2^k b_k < \infty$ then
	\[
		f(z)=e^{-F(z)}, \quad g(z)=G(z)e^{-F(z)}
	\] 
	are holomorphic on $\C\setminus (\R_-- iE_s)$. For the dense subset $\{y_\tau\}_{\tau \in \I}$ of $E_s$ we have  
	\[
		g(z)=0 \text{ for } z= -iy_\tau + e^{-\frac{m\pi-\frac{\pi}{2}}{b_k}}, \tau=(k,l) \in \I, m \in \N.
	\]
\end{corollary}
% subsection holomorphy (end)

\subsection{$C^\infty$-extension} % (fold)
\label{sub:extension}
In this section we will show that one can choose sequences $a_k,b_k,\alpha_k$ appropriately (satisfying the conditions of corollary \ref{cor_E:3.3}) such that $f,g$ are holomorphic on $\C_+$ and admit a $C^\infty$-extension to $\overline{\C_+}=\overline{\{z \in \C \colon \Re(z)>0\}}$).\\

Firstly we check that the building blocks, $a,b$, introduced in \eqref{eq:0.101}, admit such a $C^\infty$-extension to $\overline{\C_+}$ and are vanishing to infinite order in $0$ i.e.
\begin{equation}\label{eq:3.201}
	\lim_{\substack{ \abs{z} \searrow 0 \\ z \in \overline{\C_+}}} \Abs{ \frac{d^m}{dz^m} a(z)}, \Abs{ \frac{d^m}{dz^m} b(z)} = 0.
\end{equation}
By induction one shows that there are constants $C=C(m),D=D(m) >0$ and $\mu=\mu(m), \nu=\nu(m) \in \R$  (depending only on $m$) s.t. for any $0 < \alpha <1$, $z=re^{i\theta} \in \C\setminus \R_-$, $r<1$
\[
	\Abs{\frac{d^m}{dz^m}e^{-z^{-\alpha}}} \le C\Abs{z^{-2m}} \abs{e^{- z^{-\alpha}}} = C r^{-2m} e^{-\Re(z^{-\alpha})}
\]
and using \eqref{eq:cos_expansion_evalutated} and the equivalence for $\sin(z), z\in \C$
\[
	\Abs{ \frac{d^m}{dz^m} \cos(\ln(z))}= \Abs{ \mu\, \frac{\cos(\ln(z))}{z^m}+\nu \,\frac{\sin(\ln(z))}{z^m}} \le D\, r^{-m} \cosh(\theta).
\]
Hence \eqref{eq:3.201} holds if $r^{-m} e^{-\Re(z^{-\alpha})} \to 0$ as $r\to 0$ for every $m \in \N$. This is equivalent to $\Re(z^{-\alpha})+ m \ln(r) \to +\infty$ as $r \to 0$. For $z \in \overline{\C_+}\setminus\{0\}$ we have $-\frac{\pi}{2}\le \theta \le \frac{\pi}{2}$ and so for $r\to 0$ we have 
\[
	\Re(z^{-\alpha}) + m \ln(r)= r^{-\alpha}\cos(\alpha \theta) + m \ln(r)\ge r^{-\alpha} \cos(\alpha\frac{\pi}{2}) + m \ln(r) \to \infty.
\]

Similarly we can conclude the extension for $f,g$:
\begin{lemma}\label{lem_E:3.4}
	Let the sequences be $a_k=b_k= \frac{2^{-k}}{k^2}$ and 
	\[
		\alpha_k=\begin{cases}
			\alpha, &\text{ if } 0 < s <1 \text{ for some } s<\alpha <1.\\
			1- \frac{1}{2} k^{-\frac{1}{3}} &\text{ if } s=1
		\end{cases}
	\]
	Then the function $f,g$ of corollary \ref{cor_E:3.3} are holomorphic on $\C\setminus (\R_--iE_s )$ and admit $C^\infty$ extensions to $\overline{\C_+}$ with
	\begin{equation*}
		\lim_{\substack{ \dist(z,-iE_s) \to 0 \\ z \in \overline{\C_+}}} \Abs{ \frac{d^m}{dz^m} f(z)}, \Abs{ \frac{d^m}{dz^m} g(z)} = 0.
	\end{equation*}
\end{lemma}

\begin{proof}
That $f,g$ are well-defined and holomorphic is the content of corollary \ref{cor_E:3.3}. It remains to check the $\C^\infty$-extension.\\
Due to the general Leibnitz rule $\frac{d^m}{dz^m} f(z) = \sum_{n=0}^m \binom{m}{n} G^{(m-n)}(z) (e^{-F(z)})^{(n)}$ it is sufficient to check that for any $m, n \in \N$,
\[ \lim_{\substack{ \dist(z,-iE_s) \to 0 \\ z \in \overline{\C_+}}} \abs{G^{(m)}(z) (e^{-F(z)})^{(n)}} =0. \]
%Recall that if $f_k$ is a sequence of holomorphic functions on $\Omega \subset \C$ with $\sum_{k}\norm{f_k}_\infty < \infty$ then $F=\sum_{k} f_k$ is holomorphic and $\frac{d^mF}{dz^m}= \sum_{k} \frac{d^mf_k}{dz^m}$.
Firstly we note that $F$ is holomorphic on $\C_+$,  $(e^{-F(z)})'=-F'(z) e^{-F(z)}$ and
\[ \abs{F^{(m)}(z)} \le \sum_{\tau \in \I} a_k \left\lvert\frac{d^m}{dz^m} (z+iy_\tau)^{-\alpha_k}\right\rvert \le m! d^{-m-1} \sum_{k=1}^\infty a_k 2^k= \frac{\pi^2}{6}  m! d^{-m-1}\]
for $z \in \C_+, \dist(z,-iE_s)\ge d$, so that by induction and the Leibniz rule we deduce 
%$\abs{\frac{d^m}{dz^m} (z+ie_\sigma)^{-\alpha_k}}= (m+1)! r_\sigma^{-\alpha_k} \le (m+1)! d^{-m-1}$ for $z \in \{ dist(z,-iE_s)\ge d \}$ and $\frac{d}{dz} e^{-F(z)}= e^{-F(z)}\sum_{\sigma\in \I} a_k \frac{d}{dz} (z+ie_\sigma)^{-\alpha_k}$. By induction we conclude that for every $m\in \N$ there is a constant $C$ depending only on $m$ and $\sum_{k=1}^\infty 2^k a_k= \sum_{k=1}^\infty k^{-2} = \frac{\pi^2}{6}$, $0<\alpha_k\le 1$ that
\begin{equation}\label{eq:3.202}
	\Abs{\frac{d^m}{dz^m} e^{-F(z)}} \le C d^{-m-1} \abs{e^{-F(z)}} \text{ for } z \in \{\dist(z,-iE_s)\ge d \}.
\end{equation}
for a constant $C>0$ that depends only on $m$. % and $\sum_{k=1}^\infty a_k 2^k = \frac{\pi^2}{6}$.
Secondly, Cauchy's integral formula \[G^{(m)}(z)= \frac{m!}{2\pi i} \oint_{\partial B_d(z)} \frac{ G(w)}{(w-z)^{m+1}} dw\] applies since $G$ is holomorphic on $B_d(z)$.  Combining it with the uniform bound on $\abs{G}$ (lemma \ref{lem_E:3.2}) gives
\begin{equation}\label{eq:3.203}
	\abs{G^{(m)}(z)} \le  \frac{m!}{d^m} \sup_{w \in B_d(z)} \abs{G(w)} \le \frac{C m!}{d^m}.
\end{equation}
%Using similar calculations as in the previous lemma one has for any subset $\mathcal{J} \subset \I$ 
%\[
%	\frac{d}{dz} G_\mathcal{J}(z) =\sum_{\sigma\in \mathcal{J}} - b_k \frac{\sin(b_k\ln(z+ie_\sigma))}{z+ie_\sigma} G_{\mathcal{J}\setminus \{\sigma\}}.
%\] 
%As shown before $\abs{G_{\mathcal{J}\setminus \{\sigma\}}(z)} \le e^{\sum_{\mu \in \mathcal{J}\setminus \{\sigma\}} b_k \abs{\theta_\mu}}$. Because $\sin(b_k \ln(z+ie_\sigma)) = \sin(b_k \ln(r_\sigma))\cosh(b_k \theta_\sigma)+ i \cos(b_k \ln(r_k))\sinh(b_k \theta_k)$, we have $\abs{\sin(b_k \ln(z+ie_\sigma))} \le \cosh(b_k \theta_\sigma) \le e^{b_k \abs{\theta_\sigma}}$, because  Hence $\abs{\frac{d^m}{dz^m} \frac{\sinh(b_k \ln(z+ie_\sigma))}{z+ie_\sigma}} \le C_m d^{m-1} e^{b_k \abs{\theta_\sigma}} $ for all $m \in \N$, some constant $C_m$ only depending on $m$ and $z \in \{ dist(z, -iE_s)\ge d\}$. One shows again by induction that there is some constant $C_m$ depending only on $m$ and $\sum_{k=1}^\infty 2^k b_k =\frac{\pi^2}{6}<\infty$ that 
%\begin{equation}\label{eq:3.203}
%	\abs{\frac{d^m}{z^m} G(z)} \le C_m d^{-m} e^{\sum_{\sigma\in \I} b_k \abs{\theta_\sigma}} \le C_m d^{-m} \text{ for } z \in \{dist(z,-iE_s)\ge d \}.
%\end{equation}
Considering \eqref{eq:3.202}, \eqref{eq:3.203} and the general Leibniz rule the $C^\infty$ lemma follows if for every $m \in \N$ 
\[
	d^{-m}\abs{e^{-F(z)}} = e^{- (\Re{(F(z))}+ m \ln(d))} \to 0 \text{ for } d=\dist(z, -{iE_s}) \to 0.
\]
This is equivalent to 
\begin{equation}\label{eq:3.204}
	\Re{(F(z))}+ m \ln(d) \to +\infty \text{ as } d \to 0.
\end{equation}\\
To check it, let $z \in \overline{\C_+}$ with $d= \dist(z, -iE_s)>0$ be given. Fix $y \in E_s$ with $d=\abs{z-iy}$ and $\tau_k=(k,l) \in \I$ with $y \in E_{\tau_k}$ for each $k \in \N$. Take $k_0 \in \N$ with
\begin{equation}\label{eq:3.205}
\abs{E_{k_0+1, \cdot}}< d \le \abs{E_{k_0, \cdot}}
\end{equation}
Hence for $k \le k_0$ we have $r_{\tau_k} \le d + \abs{E_{\tau_k}} \le 2 \abs{E_{\tau_k}}$ and so
\begin{align*}
	\Re(F(z)) &= \sum_{\tau \in \I} a_k \cos(\alpha_k \theta_\tau) r_\tau^{-\alpha_k} \ge \sum_{k=1}^{k_0} a_k \cos(\alpha_k \frac{\pi}{2})  r_{\tau_k}^{-\alpha_k}\\
	&\ge \frac{1}{2} \sum_{k=1}^{k_0} a_k \cos(\alpha_k \frac{\pi}{2}) \abs{E_{\tau_k}}^{-\alpha_k}.
\end{align*}
We will consider $0<s<1$ and $s=1$ separately.\\ 
If $0<s<1$ we have $a_k \cos(\alpha_k \frac{\pi}{2}) \abs{E_{\tau_k}}^{-\alpha_k} = k^{-2} \cos(\alpha\frac{\pi}{2}) \zeta^{k}$ where $\zeta=2^{\frac{\alpha}{s}-1}>1$.  We combine this with
\[
	(\zeta -1) \sum_{k=1}^{k_0} k^{-2} \zeta^k = k_0^{-2} \zeta^{k_0+1} - \zeta + \sum_{k=1}^{k_0-1} (k^{-2}-(k+1)^{-2})\zeta^{k+1} \ge k_0^{-2} \zeta^{k_0+1} - \zeta
\]
to conclude that 
\begin{align*}
	&\Re(F(z))+ m \ln(d) \ge c k_0^{-2} \zeta^{k_0+1} + m \ln(d) - c \zeta\\
	&\ge c k_0^{-2} \zeta^{k_0+1} - \frac{m \ln(2)}{s}(k_0+1) - c \zeta \to +\infty \quad (k_0 \to \infty)
\end{align*}
where $c=\frac{\cos(\alpha\frac{\pi}{2})}{2(\zeta-1)}$. This is equivalent to \eqref{eq:3.204} since due to \eqref{eq:3.205}, $-\frac{\ln(2)}{s}(k_0+1) < \ln(d) \le -\frac{\ln(2)}{s}k_0$. \\
If $s=1$, we have 
\begin{equation}\label{eq:3.206}
	a_k\cos(\alpha_k \frac{\pi}{2}) \abs{E_{\tau_k}}^{-\alpha_k} \ge\frac{1}{2}\; \frac{2^{\frac{1}{4}k^{\frac{2}{3}}}}{k^{\frac{7}{3}} } \text{ for } k \ge 9. 
\end{equation}
\eqref{eq:3.206} holds because firstly $\abs{E_{\tau_k}}= 2^{-k - k^{2/3}}$, $\alpha_k = 1 - \frac{1}{2}k^{-\frac{1}{3}}$ and therefore
\[
	\frac{\ln(2^-k \abs{E_{\tau_k}}^{-\alpha_k})}{\ln(2)} =(1- \frac{1}{2}k^{-\frac{1}{3}})(k + k^{\frac{2}{3}}) -k = \frac12 k^{\frac23}(1-k^\frac13)\ge \frac{k^{\frac{2}{3}}}{4}  \text{ for } k \ge 8.
\] 
Secondly, $\cos(\alpha_k \frac{\pi}{2}) \ge (1-\alpha_k)=\frac{k^{-\frac{1}{3}}}{2}$ because $\cos((1-t)\frac{\pi}{2}) \ge t$ for $0\le t \le 1$.\footnote{holds ture because $\cos(x)$ is concave on $\abs{x}\le \frac{\pi}{2}$}
Similar as before we have
\begin{equation}\label{eq:3.207}
	(2^{\frac{1}{6}}-1) \sum_{k=9}^{k_0} \frac{2^{\frac{k^{\frac{2}{3}}}{4}}}{k^{\frac{7}{3}}} = \frac{2^{\frac{k_0^{\frac{2}{3}}+\frac{2}{3}}{4}} }{k_0^{\frac{7}{3}}}- \frac{2^{\frac{9^{\frac{2}{3}}}{4}}}{9^{\frac{7}{3}}} + \sum_{k=9}^{k_0-1}  \frac{2^{\frac{k^{\frac{2}{3}}+\frac{2}{3}}{4}}}{k^{\frac{7}{3}}}- \frac{2^{\frac{(k+1)^{\frac{2}{3}}}{4}}}{(k+1)^{\frac{7}{3}}} \ge \frac{2^{\frac{(k_0+1)^{\frac{2}{3}}}{4}}}{k_0^{\frac{7}{3}}}- 1,
\end{equation}
where we used that $k^{\frac{2}{3}}+\frac{2}{3}\ge (k+1)^{\frac{2}{3}}$\footnote{ by the mean value theorem we have $(k+1)^{\frac23}- k^{\frac23} \le \sup_{x \in [k,k+1]} \frac{d}{dx} x^{\frac23} \le \frac23$} to conclude that the sum in the middle is non-negative. We combine \eqref{eq:3.206} and \eqref{eq:3.207} to conclude 
\begin{align*}
	&\Re(F(z))+ m \ln(d) \ge \sum_{k=9}^{k_0} a_k \cos(\alpha_k \frac{\pi}{2}) \abs{E_{\tau_k}}^{-\alpha_k} + m \ln(d)\\
	&\ge c \frac{2^{\frac{(k_0+1)^{\frac{2}{3}}}{4}}}{k_0^{\frac{7}{3}}}- c - m\ln(2) ( k_0 +1+ (k_0+1)^{\frac{2}{3}}) \to +\infty \quad (k_0 \to \infty)
\end{align*}
where $c= \frac{1}{4(2^{\frac{1}{6}}-1)}$. As before it is equivalent to \eqref{eq:3.204} because of  \eqref{eq:3.205}, which is equivalent to $-\ln(2)(k_0+1+ (k_0+1)^\frac{2}{3}) < \ln(d) \le -\ln(2)(k_0+ k_0^\frac{2}{3})$.
\end{proof}

% subsection extension (end)

% section construction_of_the_holomorphic_functions (end)

\section{Applications}\label{sec:applications} 
\subsection{Minimal surfaces}\label{sub:minimal surfaces}
Given a holomorphic function $h$ on $\Omega \subset \C$ open, $Q\in \N$ one defines the irreducible holomorphic variety $\mathcal{V}\subset \Omega \times \C$ by
\begin{equation}\label{eq:4.111}
\mathcal{V}=\{ (z,u) \in \Omega \times \C \colon u^Q= h(z) \}.
\end{equation}
Following Federer we associate to $\mathcal{V}$ an integer rectifiable current of real dimension two denoted by $\llbracket \mathcal{V} \rrbracket$. It is given by integration over the manifold part of $\mathcal{V}$, $\mathcal{V}_{reg.}$ i.e. $\mathcal{V}_{reg.}=\{ (z,u) \colon u^Q = h(z), h(z) \neq 0 \}$.\\
Federer observed that $\llbracket \mathcal{V} \rrbracket$ is a mass-minimizing cycle, since $\mathcal{V}$, as a complex submanifold of $\C^2$ is calibrated by the K\"ahler form (Wirtinger's form).\\
If we take $h=g$, $\Omega=\C_+$ in \eqref{eq:4.111} we get the following example:

\begin{example}\label{ex:4.111}
Given $0<s\le 1$ and an integer $Q\ge2$ there is a mass-minimizing cycle $\mathcal{V} \subset \C_+ \times \C$ with the additional property that if $s<1$ then $\h^s(\overline{\mathcal{V}\setminus \mathcal{V}_{reg.}})=1$ and if $s=1$ then $\dim_\h(\overline{\mathcal{V}\setminus \mathcal{V}_{reg.}})=1$.
\end{example}
The additional property holds since $\mathcal{V}\setminus \mathcal{V}_{reg.}=\{(z,0) \in \C_+\times\C \colon G(z)=0 \}$ and therefore $\overline{\mathcal{V}\setminus \mathcal{V}_{reg.}}= \{(z,0) \in \C_+\times\C \colon G(z)=0 \} \cup -iE_s $. $\{(z,0) \in \C_+\times\C \colon G(z)=0 \}$ is countable so that the claim follows by the properties of $E_s$.

\begin{remark}\label{rem:4.112}
For two dimensional minimal surfaces in $\R^3$ R.~Ossermann had shown in \cite{Os} that true branching points can be ruled out in the interior. If the boundary curve is real analytic the existence branching points at the boundary can be ruled out as well. This was shown by R.~Gulliver and F.~Leslie in \cite{GuLe} for two dimensional surfaces in $\R^3$.

R.~Gulliver presents in \cite[Theorem 1.6]{Gu} the following example:
\begin{theorem}\label{theo:4.111}
There is a smooth minimal immersion $X(\Omega) \subset \R^3$, $\Omega\subset \C_+$ simply connected with the following property:
$X$ maps $\partial \Omega$ diffeomorphically onto a regular $C^\infty$ Jordan curve $\Gamma \subset \R^3$ and has a true branch point at $z=0 \in \Gamma$. The set of self intersections of $X$ consists of the union of an infinite sequence of disjoint real analytic arcs, each which joints two points of $\Gamma$ lying on opposite sides of the branch point.
\end{theorem}
His construction uses the Weierstrass representation with a holomorphic vector field that comes from a perturbation of the building block $a(z)= e^{-z^\alpha}$, \eqref{eq:0.101}, with  $\alpha =\frac{1}{7}$. It could be of interest to see if one can follow his analysis using one of the holomorphic functions $f$ or $g$ (lemma \ref{lem_E:1.1}) to construct a minimal immersion $X$ in $\R^3$ with $C^\infty$ boundary curve and a large set of true branching points on the boundary.
\end{remark}

% subsection minimal surfaces (end)

\subsection{Dirichlet minimizing $Q$-valued functions}\label{sub:Q-valued functions}
One of the implications of lemma \ref{lem_E:1.1} in the context of $Q$-valued functions had been stated heuristically in the introduction.

F.~Almgren developed in his pioneering work \cite{Almgren} the theory of multivalued functions to  prove a regularity result on area minimizing rectifiable currents. He introduced them as $Q$-valued functions. $Q \in \N$, fixed, indicates the number of values the function takes, counting multiplicity. We will refer to them from now on as $Q$-valued functions. We assume that the reader is familiar with the most basic definitions and results concerning the theory of $Q$-valued functions with focus on Dirichlet minimizers. We follow mainly the notation and terminology introduced by C.~De Lellis and E.~Spadaro in \cite{Lellis}. It differs slightly from Almgren's original one e.g. $(A_Q(\R^n), \G)$ denotes the metric space of unordered $Q$-tuples in $\R^n$, $W^{1,2}(\Omega, \A_Q(\R^n))$ the Sobolev space of $Q$-valued functions on a domain $\Omega \subset \R^N$. A recollection of the most general definitions and results omitting the actual proofs can be found in \cite[section 1]{H1}. C.~De Lellis and E.~Spadaro gave a modern revision of Almgren's original theory and results concerning Dirichlet minimizers in \cite{Lellis}.
  
 The holomorphic functions $f,g$ generate examples of $Q$-valued functions that are Dirichlet minimizing with respect to compact perturbations. Furthermore these examples are defined on $\R^2_+=\{(x,y) \in \R^2 \colon x>0\} \simeq \C_+$ and have "large" singular set towards the boundary. As we mentioned before the classical theory of Dirichlet minimizing $Q$-valued functions had been developed in  \cite{Almgren} and revisited with modern methods in \cite{Lellis}.
  
Before we are going to state the precise properties of the examples we recall the the definition of the singular set and related results thereafter the definition of $C^k(\Omega, \A_Q(\R^m)$ for a domain $\Omega \subset \R^n.$\\

%Concerning notation and some general facts to mulit-/$Q$-valued functions we refert to the short overview given in \cite[section 1]{H1}.\\
\emph{Definition of the singular set:}\\
Given a Dirichlet minimizer $u \in W^{1,2}(\Omega, \A_Q(\R^m))$, $\Omega \subset \R^N$ open, a point $y \in \Omega$ is called a regular point of $u$ if  $\exists U\subset \Omega$ open neighborhood of $y$, $u_i \in C^{\infty}(U, \R^m)$ harmonic with 
\[ u(x)= \sum_{i=1}^Q \llbracket u_i(x) \rrbracket \text{ for a.e. } x \in U \]
and $u_i(x) \neq u_j(x), \forall x \in U$ or $u_i \equiv u_j$. The open set (by definition) of all regular points is denoted by $\reg(u)$. $\sing(u)$ then denotes the relative closed complement $\Omega\setminus \reg(u)$.\\

An outcome of Almgrens original work is an estimate on the size of the singular set in the interior, compare \cite[Theorem 0.11]{Lellis}.

\begin{theorem}\label{theo:4.122}
$u \in W^{1,2}(\Omega, \A_Q(\R^m))$ Dirichlet minimizing has $\dim_\h(\sing(u))\le N-2$. In the case of $N=2$, $\sing(u)$ is countable.
\end{theorem}

This estimate had been improved by C.~De Lellis and E.~Spadaro, \cite[Theorem 0.12]{Lellis}.
\begin{theorem}\label{theo:4.123}
$u$ as above and $N=2$ then $\sing(u)$ consists of isolated points.
\end{theorem}

That the upper bound on the Hausdorff dimension is sharp is a consequence of the following:
\begin{theorem}\label{theo:4.124}
Let $\mathcal{V} \subset \C^N \times \C^m \simeq \R^{2N} \times \R^{2m}$ be an irreducible holomorphic variety with the property that $\exists \Omega \subset \C^N$ open, $C^1-$regular, $\mathcal{V}$ is is a $Q:1$ cover of $\Omega$ under the orthogonal projection and $\mathbf{M}(\mathcal{V} \cap (\Omega \times \C^m))<\infty$. Then $\exists \, u \in W^{1,2}(\Omega, \A_Q(\R^{2m})$ Dirichlet minimizing with $graph(u)=\mathcal{V}\cap (\Omega \times \C^m)$.
\end{theorem}
This was original be proven by Almgren, \cite[Theorem 2.20]{Almgren}. E.~Spadaro found a very elegant more elementary proof, \cite[Theorem 0.1]{Spadaro}.\\

Hence the holomorphic varieties $\mathcal{V}=\mathcal{V}_h$  defined in \eqref{eq:4.111} generate examples of Dirichlet minimizers:
\begin{equation}\label{eq:4.122}
u_h(z)=\sum_{\substack{v \in \C \\ v^Q= h(z)}} \llbracket v \rrbracket \text{ for } z \in \Omega.
\end{equation}

\emph{Definition of $C^k(\Omega, \A_Q(\R^m))$:}\\
Let $k\in \N$ and $\Omega \subset \R^N$, $u \in C^0(\Omega, \A_Q(\R^m))$ is said to be $C^k(\Omega, \A_Q(\R^m))$ if there exists a $Q$-valued map $U$,
\[ x \mapsto U_x(y) = \sum_{i=1}^Q \llbracket P^i_x(y) \rrbracket, \quad P_x^i \text{ is a polynomial with degree $\le k$}\]
such that the following properties hold
\begin{itemize}
\item[(a)] $U_x(x) = \sum_{i=1}^Q \llbracket P^i_x(x) \rrbracket = u(x)$ for all $x \in \Omega$;
\item[(b)] $P^i_x = P^j_x$ if $u_i(x) = u_j(x)$;
\item[(c)] whenever $K \subset \subset \Omega$, compact, $\delta > 0$ let 
\[ \rho_K(\delta)= \sup_{\substack{x,y \in K\\ \abs{x-y}\le \delta}} \inf_{\sigma \in \mathcal{P}_Q} \sum_{i=1}^Q\left(\sum_{\abs{\alpha}\le k} \abs{ D^\alpha P^i_y(y) - D^\alpha P^{\sigma(i)}_x(y)}  \abs{x-y}^{\abs{\alpha}-k} (k - \abs{\alpha})!\right)^2\]
then $\rho_K(\delta) \to 0$ as $\delta \to 0$. 
\end{itemize}

We want to remark, that condition (b) is not always assumed, compare \cite[Definition 3.6]{Lellis_select} and \cite[Definition 1.9]{Lellis} .\\
Let   $u_1, \dotsc, u_Q$ be a collection of single valued $C^k$-functions on $\Omega$. Then
\begin{equation}\label{eq: example of C^k-function with selection}  u(x)= \sum_{i=1}^Q \llbracket u_i(x) \rrbracket\end{equation}
defines a $Q$-valued $C^k$-function (including property (b)), if $D^\alpha u_i(x) = D^\alpha u_j(x)$ for all $\abs{\alpha}\le k$ whenever $u_i(x)=u_j(x)$. The function $U_x$ is given by
\[U_x(y)=\sum_{i=1}^Q \llbracket P_x^i(y) \rrbracket \]
where $P_x^i(y) = \sum_{\abs{\alpha}\le k} \frac{1}{\alpha!} D^\alpha u_i(x)(y-x)^\alpha$ is the $k$th-order Taylor polynomial of  $u_i$. Property (c) follows from the properties of the Taylor polynomials and (b) by the assumption on the order of contact.\\

Now we are able to state properly the properties of the examples:

\begin{corollary}\label{cor:4.121}
Let $0<s\le 1$ and an integer $ Q\ge 2$ be given, then there is  $u \in W_{loc.}^{1,2}(\R_+, \A_Q(\R^2))$, Dirichlet minimizing with respect to compact perturbations of $\overline{\R^2}$ and the additional properties
\begin{itemize}
\item[(i)] $u\tr{\partial \R^2_+} \in C^k(\partial \R^2_+, \A_Q(\R^2))$ for all $k \in \N$;\\
\item[(ii)] if $s<1$ then $\h^s(\overline{\sing(u)})=1$ and if $s=1$ then $\dim_\h(\overline{\sing(u)})=1$.
\end{itemize}
\end{corollary}

\begin{proof}[Proof of lemma \ref{cor:4.121}]
Let $0< s \le 1$ be fixed and $g(z)=G(z) e^{-F(z)}$ be the holomorphic function on $\C_+$ constructed in lemma \ref{lem_E:1.1}. 
\[ u(z)= \sum_{\substack{v \in \C \\ v^Q= g(z)}} \llbracket v \rrbracket \quad z \in \C_+\]
is Dirichlet minimizing and an element of $W^{1,2}(\Omega, \A_Q(\R^2))$ for any $C^1$-regular bounded subset $\Omega \subset \C_+$ as a consequence of theorem \ref{theo:4.124}.\\
It remains to check the $C^\infty$-regularity at the boundary and the property of the singular set.\\
We start with the regularity of the trace. By construction we had $g(z)=G(z) e^{-F(z)}$ is holomorphic on $\C \setminus ( \R_- - iE_s )$ and $g\tr{\C_+}$ has an $C^\infty$ extension to $\overline{\C^+}$. Furthermore $G(z)\neq 0$ for all $z \in \C\setminus ( \R -iE_s)$, $\abs{G(z)} < C$ uniformly on $\C \setminus ( \R_- - iE_s )$. So that for any $z_0\notin \R-iE_s$ there exists $r>0$ sufficient small such that $G(B_r(z_0))$ is contained in a holomorphic branch $\psi: G(B_r(z_0)) \to \C$ of the $Q$-th. root. $u$ is then explicitly given by 
\[ u(z)= \sum_{l=0}^{Q-1} \llbracket \xi^l \;(\psi \circ G)(z)\; e^{-\frac{1}{Q}F(z)} \rrbracket \quad \forall z \in B_r(z_0), \xi = e^{i \frac{2\pi}{Q}}. \]
Note that $(\xi^l-\xi^k)\psi\circ g(z)\neq 0$ for $k\neq l, z\in B_r(z_0)$, so that we are in the situation of \eqref{eq: example of C^k-function with selection}. The $k$-jet of $u$ is 
\[ \mathcal{U}^k_z= \sum_{l=0}^Q \llbracket ( \xi^l \,(\psi \circ g)(z),\xi^l \,(\psi \circ g)^{(1)}(z), \dotsc , \xi^l \,(\psi \circ g)^{(k)}(z) ) \rrbracket\] 
where we write $\psi\circ g(z)$ for $(\psi\circ G)(z)\, e^{-\frac{1}{Q}F(z)}$. The $C^\infty$-regularity will follow from
\begin{equation}\label{eq:4.124}
\abs{(\psi\circ g)^{(m)}(-iy)}= O(\dist(y,E_s)) \quad \text{ for all } m \in \N.
\end{equation}
The same arguments used in the proof to lemma \ref{lem_E:3.4} show that 
\[ \Abs{\frac{d^m}{dz^m} e^{-\frac{1}{Q} F(z)}} \le C\, d^{-m-1} \abs{e^{-\frac{1}{Q} F(z)}}= C \left( d^{-Q(m+1)} e^{- \Re(F(z))} \right)^{\frac{1}{Q}}\]
for all $z \in \{ \dist(z,-iE_s) \ge d \}$ and a constant $C=C(m) >0$.
Let $z \in \{\dist(z, \R-iE_s)>d \}$ be given, then $\psi \circ G$ is holomorphic on $B_d(z)$. So Cauchy's integral formula gives
\[ (\psi\circ G)^{(m)} = \frac{m!}{2\pi i} \oint_{\partial B_d(z)} \frac{\psi \circ G(w)}{(w-z)^{m+1}} \,dw\]
and therefore
\[ \abs{(\psi\circ G)^{(m)}(z)} \le \frac{m!}{d^m} \sup_{w\in B_d(z} \abs{G(w)}^{\frac{1}{Q}} \le Cm! d^{-m}. \]
We used the uniform bound on $\abs{G}$. Combining both bounds with the Leibniz rule we deduce \[\Abs{\frac{d^m}{dz^m}(\psi\circ G)(z) e^{-\frac{1}{Q} F(z)}} \le C \left( d^{-Qm} e^{-\Re(F(z))} \right)^{\frac{1}{Q}} \quad \forall z \in \{\dist(z,\R-iEs)>d\}.\]
So \eqref{eq:4.124} follows from \eqref{eq:3.204} where we showed that for any $m\in \N$
\[ \Re{(F(z))}+ m \ln(d) \to +\infty \text{ as } d \to 0.\]\\
It remains to check the properties of the singular set. By construction of u we have 
\[ \overline{\sing(u)}= \{ z \in \C_+ \colon g(z) =0 \} \cup -iE_S \] because $g$ has the property that to any $z\in -iE_s$ there exists $z_k \in \C_+, z_k \to 0$ and $g(z_k)=0$. Set $A_k=\{ z \in \C_+\colon g(z)=0, 2^k \le \Re(z) < 2^{k+1} \}$ for any $k \in \Z$. $A_k$ consists of isolated points since $g$ is holomorphic on $\C_+$ and therefore $\h^s(A_k)=0$ for all $k\in \Z$ and $s>0$. Hence we deduce
\[ \h^s(-iE_s) \le \h^s(\overline{\sing(u)}) \le \h^s(-iE_s) + \sum_{k\in\Z} \h^s(A_k) = \h^s(-iE_s).\]
\end{proof}

This example, corollary \ref{cor:4.121}, shows that the singular set can behave very badly towards the boundary. In the interior a blow-up analysis together with a Federer reduction argument is used to study the singular set, compare \cite[section 3]{Lellis} . With the following calculation we want to show that this procedure cannot directly transferred to the boundary. \\
Almgren's celebrated frequency function is the major tool to carry out the blow-up  analysis. For $u \in W^{1,2}(\Omega, \A_Q(\R^m))$ with $\Omega \subset \R^N$ open it is defined as
\begin{equation}\label{eq:4.125}
I(u,y,r)=\frac{D(u,y,r)}{H(u,y,r)}=\frac{r^{2-N} \int_{B_r(y)\cap \Omega} \abs{Du}^2 }{r^{1-N} \int_{\partial B_r(y)} \abs{u}^2}.
\end{equation}
Its essential property is, compare \cite[Theorem 3.15]{Lellis}
\begin{theorem}\label{theo:4.125}
Let $u\in W^{1,2}(\Omega, \A_Q(\R^m))$ be Dirichlet minimizing, then for any $y \in \Omega$ either $\exists 0< R<\dist(y,\partial \Omega)$ s.t. $u\tr{B_R(y)} \equiv 0$ or $r \in ]0, \dist(y, \partial \Omega)[ \mapsto I(u,y,r)$ is absolutely continuous, nondecreasing and positive.
\end{theorem}
Consequently the following limit is well-defined in the interior of $\Omega$
\begin{equation}\label{eq:4.126}
I(u,y)=\lim_{r\to 0} I(u,y,r)
\end{equation}
In the planar case C.~De Lellis and E.~Spadaro determined the spectrum of $y\mapsto I(u,y)$. If $u(x)=\sum_{i=1}^Q \llbracket u_i(x) \rrbracket$ satisfies $\sum_{i=1}^Q u_i(x)=0$ at almost every point then $I(u,y)$ takes values in the set $\{\frac{P'}{Q'} \colon P',Q' \in \N \text{ devisor free}, Q'\le Q\}\cup \{0\}$, \cite[Proposition 5.1]{Lellis}. \\
The following examples show that this may fail at boundary points. 

\begin{corollary}\label{cor:4.126}
Let $Q\ge 2, P > 0$ be two divisor free integers then there exists a Dirichlet minimizer $u \in W^{1,2}_{loc.}(\R^2_+, \A_Q(\R^2))$ with
\begin{itemize}
\item[(i)] $u\tr{\partial \R^2_+} \in C^k(\partial \R^2_+, \A_Q(\R^2))$ for all $k \in \N$;
\item[(ii)] for all $k \in \N$, $z_k=(e^{-k\pi + \frac{\pi}{2}},0)$ is a branch point of "order" $\frac{P}{Q}$ i.e. $I(u,z_k)= \frac{P}{Q}$;
\item[(iii)] $\lim_{r \to 0} I(u,0,r) = +\infty$.
\end{itemize}
\end{corollary}

\begin{corollary}\label{cor:4.127}
Let $Q>2$ be an integer, $0<s<1$ be given there is a Dirichlet minimizer $u\in W^{1,2}_{loc.} (\R^2_+, \A_Q(\R^2)$ with
\begin{itemize}
\item[(i)] $u\tr{\partial \R^2_+} \in C^k(\partial \R^2_+, \A_Q(\R^2))$ for all $k \in \N$;
\item[(ii)] $\sing(u)=\emptyset$, but $u(z)=Q\llbracket 0 \rrbracket \quad \forall z\in -iE_s$ with $\h^s(E_s)=1$ ;
\item[(iii)] $\lim_{n \to \infty} I(u,-iy_k, R_n) = + \infty$ for a countable subset $\{y_k\}_{k \in \N} \subset E_s$ and a sequence $R_n \to 0$.
\end{itemize}
\end{corollary}

Before we are give the proofs, we collect two observations to calculate energy and $L^2$-norm for multivalued functions arising from the holomorphic varieties defined in \eqref{eq:4.124}.\\
$\A_Q(\C)\simeq \A_Q(\R^2)$ enables us to define a $Q$-root "globally", i.e. an "inverse" to the holomorphic function $z \mapsto z^Q$ by
\begin{equation}\label{eq:4.127}
\Pi(w)= \sum_{v^Q=w} \llbracket v \rrbracket = \sum_{l=0}^Q \llbracket \xi^l v_0 \rrbracket
\end{equation}
for $\xi=e^{i \frac{2\pi}{Q} }$ and an arbitrary choice of $v_0\in \C$ with $v_0^Q=w$.
Furthermore we observed already before that for $y\in \Omega$ with $h(y)\neq 0$ there is an open neighborhood $U$ with $\abs{h(z) - h(y)}< \abs{h(y)}, \forall z \in U$. There is an holomorphic branch $\psi$ of the $Q$-root on $\abs{w-h(y)}<\abs{h(y)}$ so that $\Pi(w)=\sum_{l=0}^{Q-1} \llbracket \xi^l \psi(w) \rrbracket$ on $B_{\abs{h(y)}}(h(y))$ showing that $\Pi$ is continuous on all of $\C$. Furthermore
\begin{equation}\label{eq:4.128}
u(z) = \Pi\circ h(z) = \sum_{l=0}^{Q-1} \llbracket \xi^{l} (\psi \circ h)(z) \rrbracket \quad \forall z \in U.
\end{equation}
Hence $u \in C^k(U, \A_Q(\R^2))$ for all $k$ since we are in the situation mentioned in \eqref{eq: example of C^k-function with selection} with 
\begin{equation}\label{eq:4.129}
\mathcal{U}^k_z= \sum_{l=0}^{Q-1} \llbracket (\xi^l (\psi\circ h)(z), \xi^l (\psi\circ h)^{(1)}(z), \dotsc, \xi^l (\psi\circ h)^{(k)}(z))\rrbracket \quad \forall z \in U.
\end{equation}
We note that $\mathcal{U}^k$ does not depend on the particular choice of the branch. \\
As an immediate consequence of \eqref{eq:4.128} the $L^2$ norm of $u$ is given by
\begin{equation}\label{eq:4.1210}
\int_{V\cap \Omega} \abs{u}^2 = Q \int_{V\cap \Omega} \abs{h}^\frac{2}{Q}
\end{equation}
for any $V\subset \C$. The energy of $u$ on $V \cap \Omega$ due to \eqref{eq:4.129} is then
\begin{equation}\label{eq:4.1211}
 \int_{V \cap \Omega} \abs{Du}^2 = 2Q \int_{V\cap \Omega\setminus \{ h\neq 0 \}} \abs{(\psi\circ h)'}^2 = \frac{2}{Q} \int_{V\cap \Omega \setminus \{ h\neq 0 \}} \abs{h}^{\frac{2}{Q} -2} \abs{h'}^2 
\end{equation}
where $\psi$ is any local choice of a branch $\psi$ to the $Q$-root.

For instance we can use it to calculate the value of the frequency at interior branch points. 

\begin{example}\label{ex:4.121}
Let $h$ be holomorphic on $\Omega \subset \C$ and $u$ the related Dirchlet minimizer (see \eqref{eq:4.122}). Let $z_0 \in \Omega$ be a zero of order $P \ge 1$ then 
\[ I(u,z_0) = \frac{P}{Q}. \]

Since $z_0$ is a zero of order $P$, there is $k$ holomorphic on $\{ z \colon \abs{z}< \delta \} $, $k_0=k(0) \neq 0$ s.t. $h(z_0+z)=z^P k(z)$. We may assume that $\abs{k(z)}> \frac{1}{2} \abs{k_0}^2$ for all $\abs{z}<\delta$.  $h'(z_0+z)=P z^{P-1}k(z)( 1+ \frac{zk'(z)}{P k(z)}) = \frac{P}{z} h(z_0+z) (1+O(z))$ and so we may use $\abs{h}^{\frac{1}{Q}-1} \abs{h'}(z_0+z)= P \abs{z}^{\frac{P}{Q}-1} \abs{k_0}^{\frac{1}{Q}} (1+O(z))$ in \eqref{eq:4.1211} to deduce
\[ \int_{B_r(z_0)}\abs{Du}^2 = \frac{2P^2}{Q} \int_{B_r(0)} \abs{z}^{\frac{2P}{Q}-2} \abs{k_0}^{\frac{2}{Q}} (1+O(z)) = 2\pi P \abs{k_0}^{\frac{2}{Q}} r^{\frac{2P}{Q}} (1+O(r))\]
for any $0<r<\delta$. Similarly, using \eqref{eq:4.1210} we have
\[ \frac{1}{r} \int_{\partial B_r(z_0)} \abs{u}^2 = \frac{Q}{r} \int_{\partial B_r}\abs{z}^{\frac{2P}{Q}}\abs{k_0}^{\frac{2}{Q}}(1+O(z)) =  2\pi Q \abs{k_0}^{\frac{2}{Q}} r^{\frac{2P}{Q}} (1+O(r)).\]
We conclude the claim:
\[ I(u,z_0,r)= \frac{P}{Q}(1+O(r)).\]
\end{example}
For boundary points $z_0 \in \partial \Omega$ we are facing two problems to estimate $I(u,z_0,r)$ and possible limits. Firstly $r \mapsto I(u,z_0,r)$ is a priory not a monotone quantity as it is in the interior. Secondly, even restricting ourselves to minimizers of the the type \eqref{eq:4.122}, $h(z)$ does not necessarily have a convergent Taylor series at $z_0$.\\
The strategy will be to use the mean value theorem for integration in the radial variable to estimate $D(u,z_0,r)=\int_{B_r(z_0)\cap \Omega} \abs{Du}^2 $ from below by a multiple of $H(u,z_0,r)= \frac{1}{r}\int_{\partial B_r(z_0)\cap \Omega} \abs{u}^2$. The strategy is motivated by the following observation. Given a function $k$ holomorphic in a neighbourhood of $z\in \C$ and $k(z)\neq 0$, $\gamma >0$, for any $\xi=e^{i\theta}$ one has
\begin{align}\label{eq:4.1212a}
D_\xi \abs{k}^2 &= 2 \Re\left(\overline{k}k'\xi\right)= 2\abs{k}^2 \Re\left(\frac{k'}{k}\xi\right)\nonumber \\  D_\xi \abs{k}^\gamma &= \frac{\gamma}{2} \abs{k}^{\gamma-2} D_\xi\abs{k}^2 = \gamma \abs{k}^\gamma \Re\left(\frac{k'}{k}\xi\right).
\end{align}
%and so $D_\xi \abs{k}^\gamma = \frac{\gamma}{2} \abs{k}^{\gamma-2} D_\xi\abs{k}^2 = \gamma \abs{k}^\gamma \Re\left(\frac{k'}{k}\xi\right)$.
We observe that $D_\xi \abs{k}^\gamma \ge 0$ if $\Re\left(\frac{k'}{k}\xi\right)\ge 0$ and
\begin{equation}\label{eq:4.1212}\gamma\abs{k}^{\gamma-2}\abs{k'}^2 = \gamma \abs{k}^{\gamma} \Abs{\frac{k'}{k}}^2 \ge \gamma \abs{k}^\gamma \Re\left(\frac{k'}{k}\xi\right)^2 = \Re\left(\frac{k'}{k}\xi\right)\, D_\xi\abs{k}^\gamma.
\end{equation}
The strategy is illustrated in the following example:

\begin{example}\label{ex:4.122}
Let $h(z)=e^{-z^{-\alpha}}, 0<\alpha<1$ ($h(z)=a(z)$ of \eqref{eq:0.101}) in \eqref{eq:4.122}, i.e. $u(z)=\sum_{\substack{v \in \C \\ v^Q= h(z)}} \llbracket v \rrbracket$ with $z \in \Omega=\C_+$, then $u$ satisfies
\[ \lim_{R\to 0} I(u,0,R) = + \infty.\]
We will use the classic radial notation $z=re^{i\theta}$.
We define 
\[ \varphi(z) =r \Re\left(\frac{h'(z)}{ h(z)}e^{i\theta}\right) =\alpha \Re(z^{-\alpha})= \alpha r^{-\alpha} \cos(\alpha \theta).\]

Combining \eqref{eq:4.1211} with \eqref{eq:4.1212} ($h(z)\neq 0 \;\forall z \in \C_+$) gives
\begin{align*}
\int_{B_R\cap \C_+} \abs{Du}^2 &= \int_{B_R\cap \C_+} \frac{2}{Q} \abs{h(z)}^{\frac{2}{Q}-2}\abs{h'}^2 \ge \int_{B_R \cap \C_+} \frac{\varphi(z)}{r} \frac{\partial}{\partial r} \abs{h}^{\frac{2}{Q}}\\
&=\int_{-\frac{\pi}{2}}^{\frac{\pi}{2}} \int_{0}^R \varphi(re^{i\theta}) \left(\frac{\partial}{\partial r} \abs{h}^{\frac{2}{Q}}\right)(re^{i\theta}) dr d\theta
 % &\ge Q \int_{B_R \cap \C_+} \varphi(z) \frac{\partial \abs{\psi \circ h}^2}{\partial r} \,\frac{dz}{r} = Q \int_{-\frac{\pi}{2}}^{\frac{\pi}{2}} \int_0^R \varphi(re^{i\theta})\, \frac{\partial \abs{\psi \circ h}^2}{\partial r} \, dr d\theta \\
%&= Q \int_{-\frac{\pi}{2}}^{\frac{\pi}{2}} \varphi(Re^{i\theta})\, \abs{\psi \circ h}^2\, d\theta - \int_{-\frac{\pi}{2}}^{\frac{\pi}{2}} \int_0^R \frac{\partial\varphi}{\partial r}(re^{i\theta})\, \abs{\psi \circ h}^2\, dr d\theta \\
%&\ge \inf_{\abs{\theta}<\frac{\pi}{2}} \varphi(Re^{i\theta}) \frac{1}{R}\int_{\partial B_R \cap \C_+} \abs{u}^2 - \int_{-\frac{\pi}{2}}^{\frac{\pi}{2}} \int_0^R \frac{\partial\varphi}{\partial r}(re^{i\theta})\, \abs{\psi \circ h}^2\, dr d\theta.
\end{align*}
Since $\varphi(z) \ge \alpha r^{-\theta} \cos(\alpha \frac{\pi}{2}) >0$, \eqref{eq:4.1212a} implies that $\frac{\partial}{\partial r} \abs{h}^{\frac{2}{Q}}\ge 0$. Thus we apply the $1-$dimensional mean value theorem to deduce that to every $\abs{\theta} \le \frac{\pi}{2}$ there is $0<r_\theta\le R$ with 
\begin{align*}
\int_{-\frac{\pi}{2}}^{\frac{\pi}{2}} \int_{0}^R \varphi(re^{i\theta}) \left(\frac{\partial}{\partial r} \abs{h}^{\frac{2}{Q}}\right)(re^{i\theta}) dr d\theta &= \int_{-\frac{\pi}{2}}^{\frac{\pi}{2}} \varphi(r_\theta e^{i\theta}) \int_{0}^R \left(\frac{\partial}{\partial r} \abs{h}^{\frac{2}{Q}}\right)(re^{i\theta}) dr d\theta\\
&\ge \alpha R^{-\alpha} \cos(\alpha \frac{\pi}{2}) \int_{-\frac{\pi}{2}}^{\frac{\pi}{2}} \abs{h}^{\frac{2}{Q}}(Re^{i\theta}) d\theta.
\end{align*}
(Although it is not needed for the argument that the map $\theta \mapsto \varphi(r_\theta e^{i\theta})$ is measurable, since it is sufficient that it is point wise bounded, we included a short remark below on the measurability.)  
We conclude using \eqref{eq:4.1210} that  
\[ \int_{B_R\cap \C_+} \abs{Du}^2 \ge \frac{\alpha}{Q} R^{-\alpha} \cos(\alpha \frac{\pi}{2}) \frac{1}{R}\int_{\partial B_R\cap \C_+} \abs{u}^2\]
i.e. $I(u,0,R) \ge \frac{\alpha}{Q} R^{-\alpha} \cos(\alpha \frac{\pi}{2}) \to +\infty \;(R \to 0)$.
\end{example}
As we mentioned in the proof we give a short comment concerning the measurability. 
\begin{remark}\label{rem:meas}
We will prove the following claim:\\
\textit{Let $\mu$ be a Borel regular measure on a path-connected space $X$, $\nu$ a measure on some space $Y$ and $\mu \times \nu$ the product measure on $X\times Y$. Given $f,g$ with the properties that 
\begin{itemize}
\item[(i)] $f$, $g$, $fg$ are $\mu\times \nu$ summable, i.e. $f,g, fg \in L^1(X\times Y, \mu \times \nu)$ ;
\item[(ii)] $x \mapsto f(x,y)$ is continuous for a.e. $y$ and $g\ge 0$.
\end{itemize}
Then there exists a map $\chi: Y \to X$ s.t.
\begin{align}\label{eq:meas1}
&y \mapsto f(\chi(y),y) \, \int_{X} g(x,y) \, d\mu(x) = \int_{X} fg(x,y)\, d\mu(x) \text{ is $\nu$-integrable and }\\ \label{eq:meas2}
&f(\chi(y),y) \, \int_{X} g(x,y) \, d\mu(x) = \int_{X} fg(x,y)\, d\mu(x) \text{ for a.e. $y$}
\end{align}}

Indeed, let $A\subset Y$ be the set of $y \in Y$ s.t.
\begin{itemize}
\item[(a)] $x\mapsto f(x,y)$ is continuous and $\abs{f}$ is finite;
\item[(b)] $x \mapsto g(x,y), fg(x,y)$ are $\mu$-summable ($g(\cdot, y), fg(\cdot, y) \in L^1(X, \mu)$).\\
\end{itemize}
We have $\nu(Y \setminus A)=0$ since (a) holds for a.e. $y$ by assumption and (b) holds for a.e. $y$ by general measure theory. The $1$-dimensional mean value theorem tells that for $y \in A$ there exists $\chi(y) \in X$ s.t. the identity \eqref{eq:meas2} holds. Indeed let $y \in A$ be fixed, then $z \mapsto f(z,y) \int_X g(x,y) \, d\mu(x)$ is continuous and since $\Abs{ \int_X f(x,y)g(x,y) \, d\mu(x) } < \infty$ we can find $x_0, x_1 \in X$ s.t.
\begin{align*}& \inf_{z \in X} f(z,y) \int_X g(x,y) \, d\mu(x) \le f(x_0,y) \int_X g(x,y) \, d\mu(x)\\
& \le \int_X f(x,y) g(x,y) \, d\mu(x)\\
& \le f(x_1,y) \int_X g(x,y) \, d\mu(x) \le \sup_{z \in X} f(z,y) \int_X g(x,y) \, d\mu(x).
 \end{align*}
By assumption there is a continuous path $\gamma$ connecting $x_0$ with $x_1$. Now we may apply the $1$-dimensional mean value theorem to $t \mapsto f(\gamma(t),y) \int_X g(x,y)\,d\mu(x)$ to find a point $\chi(y)$. Since $\int_{X} (fg)(x,y) \,d\mu(x)$ is $\nu$-integrable and for all $y\in A$  \eqref{eq:meas2} is satisfied \eqref{eq:meas1} holds. If in addition $\int_X g(x,y) \,d\mu(x) \neq 0$ for a.e. $y$ then $y\mapsto f(\chi(y),y)$ is $\nu$-measurable.
\end{remark}

\begin{proof}[Proof of corollary \ref{cor:4.126}]
We claim that the minimizer $u(z)=\sum_{\substack{v \in \C \\ v^Q= b^P(z)}} \llbracket v \rrbracket$ with $b(z)= \cos(\ln(z)) e^{-z^{-\alpha}}$ (compare \eqref{eq:0.101}) has the desired properties. \\
(i) follows from the same arguments presented in the proof of corollary \ref{cor:4.121} so we omit the details here.\\
(ii) corresponds to example \ref{ex:4.121}. Since $\{ z\in \C_+ \colon b(z) =0\}= \{ e^{\frac{\pi(2k+1)}{2}} \colon k \in \Z\}$, $b'(e^{\frac{\pi(2k+1)}{2}})=(-1)^{k+1} e^{-\frac{\pi(2k+1)}{2}-e^{-\alpha\frac{\pi(2k+1)}{2}}}\neq 0$ and so  $e^{-\frac{\pi(2k+1)}{2}}$ is a zero of order $P$ to $b(z)^P$.\\
(iii) remains to be proven. We want to do it similarly to the example \ref{ex:4.122}. As before we define 
\[\varphi(z) = \Re\left( \frac{b'(z)}{b(z)} z\right)= \Re\left(\alpha z^{-\alpha} - \frac{\sin(\ln(z))}{\cos(\ln(z))}\right).
\]
$\Re(\tan(\ln(re^{i\theta})))$ is not uniformly bounded as $\abs{\theta}\to 0$, hence we can not conclude directly $\varphi(re^{i\theta})\ge 0$ for $r>0$ sufficient small. But $\abs{\tan(\ln(re^{i\theta}))}^2 \le \frac{1}{\tanh(\theta)^2}$ \footnote{ using the expansions for $\cos(x+iy)$, \eqref{eq:cos_expansion} and $\sin(x+iy)=\cos(x)\sinh(y)+ i \sin(x)\cosh(y)$ we have $\abs{cos(x+iy)}^2 \ge \sinh(y)^2, \abs{\sin(x+iy)}^2 \le \cosh(y)^2$. Combining both gives the claimed bound.} is bounded on $\frac{\pi}{4}\le \abs{\theta} \le \frac{\pi}{2}$ and so 
\begin{equation}\label{eq:4.1213}
\varphi(re^{i\theta}) \ge \alpha r^{-\alpha}\cos(\alpha \frac{\pi}{2}) - \frac{1}{\tanh(\frac{\pi}{4})} \ge 0
\end{equation}
for $\frac{\pi}{4} \le \abs{\theta}\le \frac{\pi}{2}$ and $0<r\le R$, $R>0$ sufficient small. The map
\[
\lambda \mapsto \abs{b(re^{i\lambda\theta})}^2=\abs{\cos(\ln(re^{i\lambda\theta}))}^2 e^{-2r^{-\alpha}\cos(\alpha\lambda \theta)}\]
as a product of two monotone increasing functions is monoton increasing on $\abs{\lambda \theta}\le \frac{\pi}{2}$\footnote{ $\lambda \mapsto e^{-r^{-\alpha}\cos(\alpha \lambda \theta)}$ is monotone increasing since $\cos(\alpha \lambda \theta)$ is decreasing on $\abs{\lambda \alpha \theta}\le\frac{\pi}{2}$; by \eqref{eq:cos_expansion_evalutated} one has $\abs{\cos(\ln(re^{i\\lambda theta})))}^2=\cos(\ln(r))^2\cosh(\lambda \theta)^2 + \sin(\ln(r))^2\sinh(\lambda \theta)^2$. Differentiating gives   $\frac{\partial}{\partial \lambda} \abs{\cos(\ln(re^{i\lambda\theta}))}^2 = \sinh(2\lambda \theta)\theta \ge 0$. } We can combine it with \eqref{eq:4.1210} ( $\abs{h}^{2}= \abs{b}^{2P}$ ) to
\begin{align}\label{eq:4.1214}
&\frac{1}{R} \int_{\partial B_R \cap \C_+} \abs{u}^2 = Q \int_{-\frac{\pi}{2}}^{\frac{\pi}{2}}  \abs{b(Re^{i\theta})}^{\frac{2P}{Q}} \,d\theta
\le Q \int_{\frac{\pi}{4} < \abs{\theta} < \frac{\pi}{2}} \abs{b(Re^{i\theta})}^{\frac{2P}{Q}}\,d\theta \\\nonumber 
&+ Q \int_{\abs{\theta} <\frac{\pi}{4}} \abs{b(Re^{i(\theta+\frac{\pi}{4})})}^{\frac{2P}{Q}}\,d\theta = 2Q \int_{\frac{\pi}{4} < \abs{\theta} < \frac{\pi}{2}} \abs{b(Re^{i\theta})}^{\frac{2P}{Q}}\,d\theta.
\end{align}
We use \eqref{eq:4.1211} together with \eqref{eq:4.1212} and $h= b^P$, $h'=P b^{P-1}b'$, $\abs{h}^{\frac{2}{Q}-2}\abs{h'}^2 = P^2 \abs{b}^{\frac{2P}{Q}-2}\abs{b'}^2$ to obtain
\begin{align*}
\int_{B_R\cap \C_+} \abs{Du}^2 &\ge \int_{B_R\cap \{ \frac{\pi}{4}\le \abs{\theta} < \frac{\pi}{2}\}} \abs{Du}^2 = P \int_{B_R\cap \{ \frac{\pi}{4}\le \abs{\theta} < \frac{\pi}{2}\}} \frac{2P}{Q} \abs{b}^{\frac{2P}{Q}-2} \abs{b'}^2
\\&  \ge P \int_{B_R\cap \{ \frac{\pi}{4}\le \abs{\theta} < \frac{\pi}{2}\}} \frac{\varphi(z)}{r} \frac{\partial}{\partial r} \abs{b}^{\frac{2P}{Q}}.
\end{align*}
The estimate \eqref{eq:4.1213} applied to \eqref{eq:4.1212} (i.e. $\frac{\partial}{\partial r} \abs{b}^{\frac{2P}{Q}}= \frac{2P}{Q} \frac{\varphi(z)}{r} \abs{b}^{\frac{2P}{Q}}$) shows that $\frac{\partial}{\partial r} \abs{b}^{\frac{2P}{Q}}(re^{i\theta}) \ge 0$ for $\frac{\pi}{4} \le \abs{\theta} \le \frac{\pi}{2}, 0<r<R$, and $R>0$ sufficient small. Hence we apply the $1-$dimensional mean value theorem to deduce that to every $\frac{\pi}{4} \le \abs{\theta} \le \frac{\pi}{2}$ there is $0<r_\theta\le R$ with
\begin{align*}
&\int_{B_R\cap \{ \frac{\pi}{4}\le \abs{\theta} < \frac{\pi}{2}\}} \frac{\varphi(z)}{r} \frac{\partial}{\partial r} \abs{b}^{\frac{2P}{Q}} = \int_{\frac{\pi}{4} \le \abs{\theta} \le \frac{\pi}{2}} \varphi(r_\theta e^{i\theta}) \int_{0}^R \frac{\partial}{\partial r} \abs{b}^{\frac{2P}{Q}}(re^{i\theta} \,dr d\theta\\& \ge \left( \alpha R^{-\alpha}\cos(\alpha \frac{\pi}{2}) - \frac{1}{\tanh(\frac{\pi}{4})}\right) \int_{\frac{\pi}{4} \le \abs{\theta} \le \frac{\pi}{2}} \abs{b}^{\frac{2P}{Q}}(Re^{i\theta}) \, d\theta.
\end{align*}
(Again we can avoid measurability questions using the bound \eqref{eq:4.1213}, nonetheless compare the previous remark \ref{rem:meas}.)
Recall \eqref{eq:4.1214} to deduce (iii) in total since for $R>0$ sufficient small 
\[
I(u,0,R) \ge \frac{P}{2Q} \left( \alpha R^{-\alpha}\cos(\alpha \frac{\pi}{2}) - \frac{1}{\tanh(\frac{\pi}{4})}\right)  \to \infty \quad(R\to 0).
\]
\end{proof}

\begin{proof}[Proof of corollary \ref{cor:4.127}]
We claim that for the choice  $f(z)= e^{-F(z)}$ of lemma \ref{lem_E:1.1} with a fixed $0<s<1$ the minimizer $u(z)=\sum_{\substack{v \in \C \\ v^Q= f(z)}} \llbracket v \rrbracket$ has the desired properties. \\
(i) follows as before by similar arguments presented in the proof to corollary  \ref{cor:4.121} and so we omit the details.\\
(ii) corresponds with $f(z) \neq 0$ for all $z \in \C_+$.\\
(iii) remains to be proven. We define 
\begin{align*}
R_n&=\abs{E_{n,\cdot}}+\frac{2}{3} \left(\abs{E_{n-1,\cdot}} - 2 \abs{E_{n,\cdot}}\right)=\frac{2}{3} \abs{E_{n-1,\cdot}} - \frac{1}{3} \abs{E_{n,\cdot}}=\frac{1}{3}(2^{1+\frac{1}{s}}-1) 2^{-\frac{n}{s}}\\
\underline{R}_n&= \abs{E_{n,\cdot}}+\frac{1}{3} \left(\abs{E_{n-1,\cdot}} - 2 \abs{E_{n,\cdot}}\right)=\frac{1}{3} \abs{E_{n-1,\cdot}} + \frac{1}{3} \abs{E_{n,\cdot}}=\frac{1}{3}(2^{\frac{1}{s}}+1) 2^{-\frac{n}{s}}
\end{align*}
and set $\delta= \frac{1}{3} ( \frac{\abs{E_{n-1, \cdot}}}{\abs{E_{n, \cdot}}}-2 )=\frac{1}{3}(2^{\frac{1}{s}}-2)>0$. We will show that (iii) holds for the countable set $\{y_\tau\}_{\tau \in \I}$ and the sequence $R_n$.\\
Let $y_{\tau_0}$ be given and fixed from now on. Set
\[ \I_0=\{\tau \in \I \colon y_\tau = y_{\tau_0}\}; \]
hence for any $!\exists k_0 \in \N$ s.t. $\forall \tau=(k,l)$ with $k < k_0$, $y_\tau \neq y_{\tau_0}$ and $\forall k > k_0$ $!\exists \tau=(k,l) \in \I_0$. We may assume that $\tau_0=(k_0,l_0)$. We partition $\I \setminus \I_0$ as follows:
\[ \I_1=\{\tau \in \I \colon y_\tau \notin E_{\tau_0}\}\]
and for any $\tau=(k,l) \in \I_0\setminus\{\tau_0\}$ (i.e. $l$ is odd and $k>k_0$) set
\[ \I_\tau=\{\tau' \in \I \colon y_{\tau'} \in E_{k,l+1} \cap E_{\tau_0}\}.\]
Observe that then for each such $\tau=(k,l) \in \I_0$, $\tilde{k}\ge k >k_0$ one has
\[ \abs{\{\tau'=(k',l') \in \I_\tau \colon k'=\tilde{k}\}}= 2^{\tilde{k}-k}. \]
Define
\[ \varphi(z+iy_{\tau_0})=\Re(-F'(z)\, (z+iy_{\tau_0})). \]
To simplify notation we will set $r=r_{\tau_0}, \theta=\theta_{\tau_0}$ i.e. $z+iy_{\tau_0}= re^{i\theta}$. In this case \eqref{eq:4.1212} corresponds to 
\begin{equation}\label{eq:4.1215}
\frac{\partial}{\partial r} \abs{f}^{\frac{2}{Q}}= \frac{2}{Q} \frac{\varphi(re^{i\theta})}{r} \abs{f}^{\frac{2}{Q}}.
\end{equation}
Recall from lemma \ref{lem_E:1.1} that $\frac{-1}{\alpha} F'(z)(z+iy_{\tau_0}) = \sum_{\tau \in \I} a_k (z+iy_\tau)^{-\alpha -1}(z+iy_{\tau_0})$ converging absolutely and $\Re((z+iy_\tau)^{-\alpha-1}(z+iy_{\tau_0})) = r_\tau^{-\alpha-1} r \cos((\alpha+1)\theta_\tau - \theta)$.\\
For $\tau \in \I_0$ we have $z+iy_\tau = z+iy_{\tau_0}=re^{i\theta}$ and so
\[ \Re\left( \sum_{\tau \in \I_0} a_k (z+iy_\tau)^{-\alpha -1}(z+iy_{\tau_0}) \right) = r^{-\alpha}\cos(\alpha \theta) \sum_{\tau \in \I_0} a_k \ge c_0 r^{-\alpha} \]
with $c_0 = \cos(\alpha \frac{\pi}{2}) \sum_{k=k_0}^\infty a_k>0$. \\
For $\tau \in \I_1$, $0<r<R$, $R>0$ sufficient small we have $r_\tau \ge \delta \abs{E_{k_0,\cdot}}$ because $r_\tau \ge \abs{E_{k_0-1,\cdot}} -2 \abs{E_{k_0,\cdot}}-r$. Therefore we found
\[ \Re\left( \sum_{\tau \in \I_1} a_k (z+iy_\tau)^{-\alpha -1}(z+iy_{\tau_0}) \right) \ge - (\delta \abs{E_{k_0, \cdot}})^{-\alpha-1} r \sum_{\tau \in \I_1} a_k \ge -c_1 r\]
%where we used that $\Re((z+iy_\tau)^{-\alpha-1}(z+iy_{\tau_0})) = r_\tau^{-\alpha-1} r \cos((\alpha+1)\theta_\tau - \theta) \ge - (\delta \abs{E_{k_0, \cdot}})^{-\alpha-1} r$.\\
In the rest of the argument we restrict us to $\underline{R}_n \le r\le R_n$ and $n>N$ for some large $N\in \N$. If $\tau=(k,l) \in \I_0$ with $k_0 <k \le n$ and $\tau' \in \I_\tau$ then $r_{\tau'} \ge \abs{y_{\tau'}- y_\tau} -r \ge \abs{E_{k-1,\cdot}}-\abs{E_{k,\cdot}}-R_n\ge \delta \abs{E_{k,\cdot}}$, so that \footnote{we use the simple estimate $\sum_{l\ge k} \frac{2^{-l}}{l^2}\le \frac{2^{-k}}{k^2} \sum_{l=0}^\infty 2^{-l} \le 2\frac{2^{-k}}{k^2}$}
\begin{align*}
&\sum_{\substack{\tau=(k,l) \in \I_0\\ k_0 < k \le n}} \sum_{\tau'\in \I_\tau} a_{k'} r_{\tau'}^{-\alpha-1}r \cos((\alpha+1)\theta_{\tau'} -\theta) \ge - \sum_{k_0<k\le n} (\delta \abs{E_{k,\cdot}})^{-\alpha-1} r \sum_{k'=k}^\infty \frac{2^{-k'}}{(k')^2}\\
& \ge - \frac{2r}{\delta^{\alpha+1}} \sum_{k_0<k\le n} \frac{M^k}{k^2} \ge - \frac{2r}{\delta^{\alpha+1}} \frac{M^{n+1}-M^{k_0+1}}{k_0^2(M-1)} \ge - \frac{c'_2}{k^2_0} r M^n
\end{align*}
where $M=(2^{\frac{\alpha+1}{s}-1})>1$ and so $2^{-k}\abs{E_{k, \cdot}}^{-\alpha-1}=M^k$. 
If $\tau=(k,l) \in \I_0$ with $n <k$ and $\tau' \in \I_\tau$ then $r_{\tau'} \ge r-  \abs{y_{\tau'}- y_\tau} \ge \underline{R}_n - \abs{E_{k-1,\cdot}} \ge \underline{R}_n - \abs{E_{n, \cdot}} = \delta \abs{E_{n, \cdot}}$ hence
\begin{align*}
&\sum_{\substack{\tau=(k,l) \in \I_0\\  n < k}} \sum_{\tau'\in \I_\tau} a_{k'} r_{\tau'}^{-\alpha-1}r \cos((\alpha+1)\theta_{\tau'} -\theta) \ge - (\delta \abs{E_{n, \cdot}})^{-\alpha-1} r \sum_{k=n+1}^\infty \sum_{k'=k}^\infty \frac{2^{-k}}{(k')^2}\\
& \ge  - (\delta \abs{E_{n, \cdot}})^{-\alpha-1} r\sum_{k=n+1}^\infty 2\frac{2^{-k}}{k^2} \ge - r\frac{2}{\delta^{\alpha+1}n^2} M^n = - \frac{c_2''}{n^2} rM^n.
\end{align*}
Summarising for $\underline{R}_n \le r \le R_n$ and $n\ge N=N(k_0)$, we have
\begin{equation}\label{eq:4.1216}
\frac{1}{\alpha} \varphi(re^{-i\theta}) \ge r^{-\alpha} \left( c_0 - c_1 r^{1+\alpha} - \left( \frac{c_2'}{k^2_0} + \frac{c_2''}{n^2} \right) M^n r^{1+\alpha},\right) \ge \frac{c_0}{2} r^{-\alpha}
\end{equation}
because $M^nr^{1+\alpha} \le M^n R_n^{1+\alpha}= \left(\frac{1}{3}(2^{1+\frac{1}{s}}-1)\right)^{1+\alpha} \, 2^{-n} \to 0$ (as $n \to \infty$).\\
\eqref{eq:4.1215} and \eqref{eq:4.1216} gives for $\underline{R}_n \le r \le R_n$
\[ \frac{\partial}{\partial r} \ln(\abs{f}^{\frac{2}{Q}}(-iy_{\tau_0} + re^{i\theta})) = \frac{2\alpha}{Q} \varphi(re^{i\theta}) \ge \frac{c_0 \alpha}{Q} r^{-\alpha} \]
or integrated
\begin{equation}\label{eq:4.1217}
\ln\left(\frac{\abs{f}^{\frac{2}{Q}}(-iy_{\tau_0}+R_ne^{i\theta})}{\abs{f}^{\frac{2}{Q}}(-iy_{\tau_0}+\underline{R}_n e^{i\theta})} \right)  \ge c R_n^{-\alpha}
\end{equation}
with $c= \frac{c_0}{Q} \left( \left(\frac{R_n}{\underline{R}_n}\right)^\alpha -1\right)>0$ (independent of $n$).\\
Now we combine the just established with \eqref{eq:4.1211}
\begin{align*}
&\int_{B_{R_n}(-iy_{\tau_0})\cap \C_+} \abs{Du}^2 \ge \int_{\{ \underline{R}_n \le\abs{z+iy_{\tau_0}}\le R_n \}\cap \C_+} \abs{Du}^2 \\& = \frac{2}{Q} \int_{\{ \underline{R}_n \le r\le R_n \}\cap \C_+} \abs{-F'}^2 \abs{f}^{\frac{2}{Q}} \ge \int_{\{ \underline{R}_n \le r\le R_n \}\cap \C_+} \frac{2}{Q} \frac{\varphi(re^{i\theta})}{r} \frac{\partial}{\partial r} \abs{f}^{\frac{2}{Q}}.
\end{align*}
As before \eqref{eq:4.1215} and \eqref{eq:4.1216} show that $\frac{\partial}{\partial r} \abs{f}^{\frac{2}{Q}} >0$ for $\underline{R}_n \le r \le R_n$. We apply as before the $1-$dimensional mean value theorem to deduce that to every $\abs{\theta} \le \frac{\pi}{2}$ there is $0<r_\theta\le R$ with
\begin{align*}
&\int_{\{ \underline{R}_n \le r\le R_n \})\cap \C_+} \frac{2}{Q} \frac{\varphi(re^{i\theta})}{r} \frac{\partial}{\partial r} \abs{f}^{\frac{2}{Q}} = \int_{-\frac{\pi}{2}}^{\frac{\pi}{2}} \varphi(r_\theta e^{i\theta}) \int_{\underline{R}_n}^{R_n} \frac{\partial}{\partial r} \abs{f}^{\frac{2}{Q}} \,dr d\theta\\&
= \int_{-\frac{\pi}{2}}^{\frac{\pi}{2}} \varphi(r_\theta e^{i\theta}) \left(\abs{f}^{\frac{2}{Q}}(-iy_{\tau_0}+ R_n e^{i\theta}) -  \abs{f}^{\frac{2}{Q}}(-iy_{\tau_0}+ \underline{R}_n e^{i\theta}) \right)\, d\theta\\
&\ge \frac{\alpha c_0}{2} R_n^{-\alpha} \left(1- e^{-cR_n^{-\alpha}} \right) \int_{-\frac{\pi}{2}}^{\frac{\pi}{2}} \abs{f}^{\frac{2}{Q}}(-iy_{\tau_0}+ R_n e^{i\theta}) \,d\theta %= \frac{\alpha c_0}{2Q} R_n^{-\alpha} \left(1- e^{-cR_n^{-\alpha}} \right) \frac{1}{R_n} \int_{\partial B_{R_n}(y_{\tau_0})\cap \C_+} \abs{u}^2.
\end{align*}
(With the same observations as before, we can avoid measurability questions by \eqref{eq:4.1216}.) 
We used in the last line \eqref{eq:4.1216} and \eqref{eq:4.1217}.  Finally remembering \eqref{eq:4.1210} we conclude (iii) since we found
\[I(u,-iy_{\tau_0}, R_n) \ge \frac{\alpha c_0}{2Q} R_n^{-\alpha} \left(1- e^{-cR_n^{-\alpha}} \right) \to \infty \quad (\text{as }n \to \infty). \]
\end{proof}

\subsection{Unique continuation}
Consider an elliptic operator $L$ in divergence form
\begin{equation}\label{eq:4.131}
Lu=D_i( a^{ij}(x) D_ju) + b^i(x)D_iu + c(x) u.
\end{equation}
A function $u \in L_{loc.}^2(\Omega)$ is said to vanish of infinite order at a point $x_0\in \overline{\Omega}$ if 
\begin{equation}\label{eq:4.132}
\int_{B_R(x_0) \cap \Omega} \abs{u}^2 = O(R^k) \text{ for every } k \in \N.
\end{equation}
An elliptic operator $L$ as in \eqref{eq:4.131} is said to have the strong unique continuation property in $\Omega$ if the only $H^{1,2}_{loc.}(\Omega)$ solution of $Lu=0$ on $\Omega$ which vanishes of infinite order at a point $x_0\in \Omega$ is $u \equiv 0$.\\
N.~Garofalo, F.~Lin showed in \cite[Theorem 1.1]{GaLi} that $L$ has the unique continuation property under certain assumptions on the regularity and ellipticity of the coefficients $a^{ij}(x), b^i(x),c(x)$. They are able to deduce their result proving a doubling theorem like the following, which the prove using the frequency function.
(The quoted version can be found in \cite[Theorem 6.1]{CoMi})
\begin{theorem}\label{theo:4.131}
Let $L$ as in \eqref{eq:4.131} with $a^{ij}(x)$ symmetric, uniformly elliptic and Lipschitz, $b^i(x), c(x)$ continuous, then if $u \in H^{1,2}_{loc.}(B_{2R_0}(x_0))$ non constant solves $Lu=0$ on $B_{2R_0}(x_0)$ then there exists $0<R=R(a^{ij},b^i,c,x_0)<R_0$ and $\overline{d}=\overline{d}(a^{ij},b^i,c,x_0,u)>0$ s.t.
\[ \int_{B_{2r}(x_0)} u^2 \le 2^{2\overline{d}} \int_{B_r(x_0)} u^2 \quad \forall 0<r<R \]
\end{theorem}
A consequence of lemma \ref{lem_E:1.1} is that a strong unique continuation theorem fails for boundary points.

\begin{example}\label{ex:4.131}
Given $0<s\le 1$ there exists $u \in C^\infty(\overline{\R^2_+})$, $u \neq 0$ with 
\[ \Delta u = 0 \text{ on } \R^2_+ \text{ (i.e. harmonic)} \]
and a set $E_s \subset \partial \R^2_+$ with $\h^s(E_s)=1$ ($0<s<1$), $dim_{\h}(E_s)=1$ ($s=1$) such that $u$ vanishes to infinite order for all $z \in -iE_s$.
\end{example}

Observe that $\Delta$ satisfies the conditions of theorem \ref{theo:4.131} and therefore has the strong unique continuation property in the interior of $\R^2_+$. 

\begin{proof}[Proof of example \ref{ex:4.131}]
Let $0<s\le 1$ be given and $f$ the related holomorphic function of lemma \ref{lem_E:1.1}. Since $f$ is $C^\infty$ on $\overline{\C_+}$ (\ref{lem_E:3.4}) and $\overline{\C_+}$ convex we have by $1$-dimensional analysis 
\begin{equation}\label{eq:4.133}
f(z)= \sum_{l=1}^{k-1} \frac{1}{l!} f^{(l)}(z_0)(z-z_0)^l + \frac{1}{(k-1)!} \int_0^1 (1-s)^{k-1} f^{(k)}(z_0+s(z-z_0))(z-z_0)^k \, ds.
\end{equation}
The function
\[ u(z)= \Re(f(z))
\]
is harmonic and non-constant on $\R^2_+$, $\C^\infty$ on $\overline{\R^2}$ and has the desired property since for $z_0 \in -iE_s$, $f^{(l)}(z_0)=0 \forall l$ and therefore by \eqref{eq:4.133}
\[ \abs{u(z)} \le \frac{1}{k!} \sup_{ w \in \overline{\C_+} \cap B_1(z_0)} \abs{f^{(k)}(w)} \, \abs{z-z_0}^k \text{ forall } z \in \overline{\C_+} \cap B_1(z_0). \]
This implies that $u$ satisfies \eqref{eq:4.132}.
\end{proof}

%Almgren observed and proved in his big regularity paper \cite[Theorem 2.20]{Almgren} that $u \in W^{1,2}_{loc.}(\Omega, \A_Q(\R^2)$ and Dirichlet minimizing. E.~Spadaro gave an elegant new proof to it in.\\
%From a regularity point of view $\sing(u)=\{z \in \Omega \colon h(z) =0 \}$.  One uses Almgrens celebrated frequency function, given for $u \in W^{1,2}(\Omega, \A_Q(\R^N)$ Dirichlet minimizing and $z_0\in \Omega$ with $u(z_0)=Q\llbracket 0\rrbracket$
%\begin{equation}\label{eq:4.122}
%I(u,z_0,r)= \frac{ \int_{B_r(z_0)\cap \Omega} \abs{Du}^2}{\frac{1}{r} \int_{\partial B_r(z_0) \cap \Omega} \abs{u}^2}.
%\end{equation}
% to study the singular set $\sing(u)$.   
% subsection Q-valued function (end)

% section applications (end)
\bibliographystyle{plain}
\bibliography{bib-example}

\end{document}